\documentclass[12pt]{amsart}
\usepackage{amscd,amssymb,graphics}

\usepackage{amsfonts}
\usepackage{amsmath}
\usepackage{amsxtra}
\usepackage{latexsym}
\usepackage[mathcal]{eucal}
\usepackage{graphicx}
\usepackage[export]{adjustbox}
\usepackage{graphics,colortbl}

\input xy
\xyoption{all}
\usepackage{epsfig}

\usepackage[pdftex,bookmarks,breaklinks]{hyperref}
\newtheorem{thm}{Theorem}[section]

\newtheorem{cor}[thm]{Corollary}
\newtheorem{lem}[thm]{Lemma}

\newtheorem{prop}[thm]{Proposition}

\theoremstyle{definition}
\newtheorem{defin}[thm]{Definition}

\theoremstyle{remark}
\newtheorem{remark}[thm]{Remark}

\newtheorem{ex}[thm]{Example}

\newtheorem{claim}[thm]{Claim}
\newtheorem{question}[thm]{Question}
\newtheorem{problem}[thm]{Problem}

\numberwithin{equation}{section}

\newcommand{\delete}[1]{} % Comment out text.
\newcommand{\nt}{\noindent}
\newcommand{\sk}{\vskip 0.2cm}

\newcommand{\ben}{\begin{enumerate}}

\newcommand{\een}{\end{enumerate}}
\newcommand{\bit}{\begin{itemize}}

\newcommand{\eit}{\end{itemize}}

\def\R {{\mathbb R}}
\def\N {{\mathbb N}}
\def\Z {{\mathbb Z}}
\def\Q {{\mathbb Q}}

\def\T {{\mathbb T}}
\def\GL{\operatorname{GL}}

\def\E{{\mathrm{E}}}

\def\Aut{{\mathrm Aut}\,}

\def\Homeo{{\mathrm{Homeo}}\,}

\def\diam{{\mathrm{diam}}}

\def\E{{\mathcal E}}

\def\WAP{\operatorname{WAP}}
\def\Tame{\operatorname{Tame}}

\newcommand{\eva}{\rm{eva}}

\def\QED{\nobreak\quad\ifmmode\roman{Q.E.D.}\else{\rm Q.E.D.}\fi}

\newcommand{\al}{\alpha}
\newcommand{\Ga}{\Gamma}
\newcommand{\ga}{\gamma}
\newcommand{\del}{\delta}
\newcommand{\Del}{\Delta}
\newcommand{\ep}{\epsilon}
\newcommand{\sig}{\sigma}

\newcommand{\La}{\Lambda}

\newcommand{\om}{\omega}
\newcommand{\Om}{\Omega}

\newcommand{\br}{\vspace{3 mm}}

\newcommand{\cls}{{\rm{cls\,}}}

\newcommand{\supp}{{\rm{supp\,}}}
\newcommand{\tri}{\bigtriangleup}

\newcommand{\Asym}{\rm{Asym}}

\newcommand{\Id}{\rm{Id}}

\newcommand{\Fcal}{\mathcal{F}}
\newcommand{\Gcal}{\mathcal{G}}

\newcommand{\Ocal}{\mathcal{O}}

\newcommand{\cb}{\mathbf{c}}

\newcommand{\Split}{{\rm{Split}}}

\newcommand{\ch}{\mathbf{1}}

\newcommand{\card}{\rm{card\,}}

\newcommand{\Adh}{{\rm{Adh\,}}}

\newcommand{\osc}{{\rm{osc\,}}}

\newcommand{\inte}{{\rm{int\,}}}
\newtheorem{exa}[thm]{Example}

\newcommand{\rest}{\upharpoonright}
\begin{document}

\title[A hierarchy in the class of tame dynamical systems] 
{Todor\u{c}evi\'{c}' Trichotomy and a hierarchy in the class of tame dynamical systems}

%Authors
%    Information for first author
\author[]{Eli Glasner}
\address{Department of Mathematics,
Tel-Aviv University, Ramat Aviv, Israel}
\email{glasner@math.tau.ac.il}
\urladdr{http://www.math.tau.ac.il/$^\sim$glasner}

%    Information for second author
\author[]{Michael Megrelishvili}
\address{Department of Mathematics,
Bar-Ilan University, 52900 Ramat-Gan, Israel}
\email{megereli@math.biu.ac.il}
\urladdr{http://www.math.biu.ac.il/$^\sim$megereli}

\date{2021, July 7}

\begin{abstract}  Todor\u{c}evi\'{c}' trichotomy in the class of 
separable Rosenthal compacta induces a hierarchy in the class of tame 
(compact, metrizable) 
dynamical systems $(X,T)$ according to the 
topological properties of their enveloping semigroups $E(X)$. 
More precisely, we define the classes
$$
	\mathrm{Tame}_\mathbf{2} \subset \mathrm{Tame}_\mathbf{1} \subset \mathrm{Tame},
$$ 
where $\mathrm{Tame}_\mathbf{1}$ is the proper subclass of tame systems  
with first countable $E(X)$, and $\mathrm{Tame}_\mathbf{2}$ is 
its proper subclass consisting of systems with hereditarily separable $E(X)$. 
We study some general properties of these classes
and exhibit many examples to illustrate these properties.	
\end{abstract} 

\subjclass[2010]{Primary 37Bxx; Secondary 54H15, 54H05, 54F05}

\keywords{almost automorphic system, circular order, enveloping semigroup, linear order, Rosenthal compact, Sturmian system, tame dynamical system}

\thanks{This research was supported by a grant of the Israel Science Foundation (ISF 1194/19)} 

\maketitle

\setcounter{tocdepth}{1}
\tableofcontents

\section*{Introduction}
\sk 

In this work we continue our study of the following theme:

%060521  I suggest to replace "question" by "problem"
\begin{problem} \label{q:theme} 
Let $T$ be a topological group and $(X,T)$ a compact dynamical system.
	Let $E:=E(X,T)$ be the enveloping semigroup of $(X,T)$. 
	How is the topology of the compact space $E$ related to the dynamical properties of the system $(X,T)$ ?
\end{problem}

For the background and history of research in this direction
see, for example, the survey papers \cite{G-env}, \cite{GM-survey}.  
%060521 Here is a an appropriate place for some motivating words. Like what the referee writes. I give here some short option  
It is a well known phenomenon that enveloping semigroups of compact metrizable dynamical systems often are nonmetrizable huge compacts. 
Recall that by \cite{G-M-U}, 
%zzz
$E(X)$ is metrizable if and only if $(X,T)$ is hereditarily nonsensitive (HNS); see below the short explanation in 
Section \ref{s:Prel}. This is a quite restricted class of systems, 
%zzz
although it contains the weakly almost periodic (WAP) systems. 
%2005 A larger 

A much larger class is that of dynamical systems $(X,T)$ for which
%However, 
$E(X)$ is a Rosenthal compactum.
It coincides with the class of tame metrizable dynamical systems. 
This fact is a consequence of the 
 dynamical analog of the  Bourgain--Fremlin--Talagrand dichotomy theorem \cite{G-M}
 (see Theorem \ref{f:DynBFT} below). %arx
 %  Tame
 %zzz
% So, 
 Thus tame systems play a principal role
 % in these studies. 
 when we address the above problem.
 Many remarkable naturally defined systems coming from geometry, 
 analysis and symbolic dynamics are tame but not HNS. 
 
In the present paper we introduce two subclasses of tame metric dynamical systems.

\begin{defin} \label{d:Tame_k} Let $(X,T)$ be a compact metrizable dynamical system 
	%	\footnote{(with at least separately continuous action)} 
	and $E(X,T)$ be its (necessarily separable) enveloping semigroup. 
	We say that this system is:
	\ben 
	\item[(1)] tame$_\mathbf{1}$ if $E(X,T)$ is first countable; % Rosenthal compactum;   
	\item[(2)] tame$_\mathbf{2}$ if $E(X,T)$ is hereditarily separable. % Rosenthal compact space. 
	\een
	The corresponding classes will be denoted by 
	$\mathrm{Tame}_\mathbf{1}$ and $\mathrm{Tame}_\mathbf{2}$ respectively. 
	As is shown in Proposition \ref{r:incl} below, we have 
	$$
	\mathrm{Tame}_\mathbf{2} \subset \mathrm{Tame}_\mathbf{1} \subset \mathrm{Tame}.
	$$ 
\end{defin}

%In order to better motivate this definition we  recall below in Section \ref{s:Prel} some basic facts.
%% about tame systems and their  role in dynamical systems. 
%First of all mention that the idea of 
This hierarchy arises naturally from %the following Todor\u{c}evi\'{c}'s trichotomy in 
deep results of Todor\u{c}evi\'{c} 
and Argyros--Dodos--Kanellopoulos
about separable Rosenthal compacta.  
%Recall the following version of Todorcevic's result 
%(which was discussed in the survey \cite{GM-survey})

%060521  I slightly changed some words and a presentation in this thm. So, please read the whole formulation 
\begin{thm} \label{tri} \index{Todor\u{c}evi\'{c}' trichotomy} 
	(Todor\u{c}evi\'{c}' Trichotomy 
	%060521 \cite{T}, \cite{A-D-K}, \cite[Theorem 10.1]{Debs}) 
	\cite{T}, \cite[Section 4.3.7]{A-D-K})
	Let $K$ be a non-metrizable separable Rosenthal compactum. 
	Then $K$ satisfies exactly one of the following alternatives: 
	\begin{enumerate}
		\item[(0)]
		$K$ is not first countable 
		%zzz
		%\nl 
		(it then contains a copy of 
		%0711
		$A({\bf c})$, the
		Alexandroff compactification of a discrete space of size continuum).
		\item[(1)]
		$K$ is first countable but $K$ is not hereditarily separable 
		%\nl 
		(it then  contains 
%060521
either a copy of 
		$D(\{0,1\}^\N)$, the \textit{Alexandroff duplicate of
		the Cantor set}, or $\widehat{D}(S(\{0,1\}^\N))$, the \textit{extended duplicate of the split Cantor set}).
		\item[(2)]
		$K$ is hereditarily separable and non-metrizable 
		%\nl 
		(it then contains a copy of the split interval).
	\end{enumerate}
\end{thm}

\br  
%060521  Eli, I think we have two options: 1) remarking somewhere  about "incomplete" formulation of Debs; 2) not mentioning HERE Debs work at all.  
%zzz
%Note that in \cite[Theorem 10.1]{Debs} the possible embeddability of $\widehat{D}(S(\{0,1\}^\N)))$ in $K$ for item (1) of Theorem \ref{tri} was missed in the formulation.   

%We note that, by 
By results of R. Pol \cite[Section 4, Theorem 3.3]{P}, 
every hereditarily separable Rosenthal compact space is first countable (see Debs \cite{Debs}).  
(A topological space $K$ is {\em perfectly normal} if it is normal and every closed subset of $K$ is
a $G_\del$ set.)

%1411
\begin{thm}\label{Pol}
For a separable Rosenthal compactum $K$ the following conditions are equivalent:
\begin{enumerate}
\item
$K$ is perfectly normal.
\item
$K$ is hereditarily separable.
\item
$K$ contains no discrete subspace of cardinality continuum.
\end{enumerate}
\end{thm}

%	\item 
We also observe that $A({\bf c})$ is a continuous image of
%1211 $D({\bf c})$
$D(\{0,1\}^\N)$ 
(we refer to \cite{T, Debs} for the definition and discussion of these spaces).  
Thus a first countable Rosenthal compactum %as well as tame systems, 
%``do not respect the subordination". 
can admit a quotient which is not first countable.
%2810
We will have a similar situation among our tame classes; namely
%2910     subscripts \mathbf{1}  and \mathbf{2}  here and some other places 
a tame$_\mathbf{1}$ system can admit a factor which is not tame$_\mathbf{1}$
(Examples \ref{e:EllisProj},  \ref{e:Akin} 
%28a
and Lemma \ref{Y*}). 
Also, as we will see (Proposition \ref{p:AlmostAut}), the product of two tame$_\mathbf{2}$ systems
need not be tame$_\mathbf{2}$.
%2011 
Note that the class Tame$_\mathbf{2}$ is closed under subsystems and factors.   
	%%2011 because continuous onto map (of enveloping semigroups) preserves hered-separ
%2111
The class	Tame$_\mathbf{1}$ is closed under countable products
%arx 
but not under subsystems (see Examples \ref{sub} below). 
	%2011 because if $X$ is embedded into $\prod_{n \in \N} X_n$ then $E(X)$ is embedded into $\prod_{n \in \N} E(X_n)$ 
%2011  I don't know if Tame$_\mathbf{1}$ is closed under subsystems

\br    
\subsection{Preliminaries} 
\label{s:Prel} 

Recall that for any topological group $T$ and any \textit{dynamical $T$-system} $X$, defined by a 
continuous homomorphism $j \colon T \to \Homeo(X)$ into the group of homeomorphisms of the compact space $X$,  
the corresponding enveloping semigroup $E(X,T)$  (or just $E(X), E$ when $T$ and $X$ are understood)
was defined by Robert Ellis as the 
pointwise closure of the subgroup $j(T)$ of $\Homeo(X)$ in the product space $X^X$. 
One may easily modify this definition for semigroup actions. 

$E(X,T)$ is a compact right topological semigroup whose algebraic and topological structure often reflects dynamical properties of $(X,T)$. 

%2011 I suggest to add this important technical fact (we use many times $\pi_*$,see for ex. Thm 2.3)  
Let $\pi\colon X \to Y$ be a factor 
%2111
map
of $T$-systems (i.e., 
%060521 $\pi(tx)=t \pi(x) \ \forall t \in T \ \forall x \in X$). 
$\pi(tx)=t \pi(x)$ for all $t \in T$ and  $x \in X$). 
%By \cite[Prop. 3.8]{E-book} 
Then
there exists a (unique) continuous surjective semigroup homomorphism
$\pi_* \colon E(X) \to E(Y)$ such that 
$\pi \circ p=\pi_*(p) \circ \pi$ for every $p \in E(X)$
%2111
and such that $\pi_*(j_X(t)) = j_Y(t)$ for every $t \in T$.
%Since $\pi_*$ is onto, for every $p_Y \in E(Y)$ there exists
%$p_X \in E(X)$ such that the following diagram commutes.
%%%
%
%\begin{equation*}
%	\xymatrix {
%		X \ar[d]_{\pi} \ar[r]^{p_X} & X \ar[d]^{\pi} \\
%		Y \ar[r]^{p_Y} & Y }
%\end{equation*}

For 
%2011 basic definitions and some useful information 
more
%2111 
%useful 
information about dynamical 
systems and their enveloping semigroups we refer to \cite{G-env} and \cite{GM-survey}. 
Of special interest for us 
%2111
%here 
will be the successively larger classes 
of almost periodic (AP), weakly almost periodic (WAP), 
hereditarily nonsensitive (HNS), and tame dynamical systems. 
We include here some information about the class of tame systems. 

\sk
\subsection*{Tame dynamical systems} 
A real valued
%060521 
bounded 
 continuous function $f \in C(X)$ is said to be \textit{tame} if the family 
$fT:=\{f \circ t: t \in T\}$ of real valued functions has no \textit{independent} infinite sequence 
(in the sense of H. Rosenthal, 
%060521 \cite{Ro}). Notation $f \in \Tame(X)$.
\cite{Ro}); 
%zzz
%we denote such a property by $f \in \Tame(X)$.
we denote by $\Tame(X)$ the collection of these functions.

By \cite{GM-rose} $f$ is tame iff for every $p \in E(X,T)$ the function $f \circ p \colon X \to \R$ has 
the {\em point of continuity property} (PCP in short).  
The system $(X,T)$ is {\em tame} if $\Tame(X)=C(X)$.
A metric dynamical system $(X,T)$ is tame iff 
every element $p$ of the enveloping semigroup $E=E(X,T)$ 
is a limit of a sequence of elements from $T$, \cite{G-M,G-M-U},
iff $p$ is of Baire class 1. 
%1111
%Equivalent to say
Another equivalent condition is that 
${card} \; {E}(X)$ is $\leq
2^{\aleph_0}$ (Theorem \ref{f:DynBFT}). 

Tame dynamical systems were introduced by A. K\"{o}hler \cite{Koh} under the 
name \textit{regular systems}. The term ``tame" was proposed later in \cite{G-06}.

Tame dynamical systems form a class of low complexity dynamical systems 
where several remarkable dynamical and topological properties meet.  
This class is quite large and is closed under subsystems, products and factors.  
By the dynamical analog of the Bourgain--Fremlin--Talagrand theorem, a compact 
metrizable dynamical system is tame if and only if its 
(always, compact and separable) enveloping semigroup $E(X,T)$ is a  
Rosenthal compact space. More precisely we have:

\begin{thm} \label{f:DynBFT} 
	\cite{G-M} \emph{(A dynamical version of BFT dichotomy)}
	Let $(X,T)$ be a compact metric dynamical system and let $E$ be its
	enveloping semigroup. 
	Either
	\begin{enumerate}
		\item
		$E$ is a separable Rosenthal compact space (hence $E$ is a Fr\'{e}chet-Urysohn space and ${card} \; {E} \leq
		2^{\aleph_0}$); or
		\item
		%the compact space 
		$E$ contains a homeomorphic
		copy of $\beta\N$ (hence ${card} \; {E} = 2^{2^{\aleph_0}}$).
	\end{enumerate}
	The first possibility
	holds if and only if $(X,T)$ is a tame system.
\end{thm}

%2510 As we already mentioned before, the 
The metrizable tame systems are exactly those systems which admit 
a representation on a separable Rosenthal Banach space \cite{GM-rose}
(a Banach space is called \emph{Rosenthal} if it does not contain an isomorphic copy of $l_1$). 
An important subclass of the class of Rosenthal Banach spaces is the class of Asplund Banach spaces. 
(A Banach space $V$ is {\em Asplund} iff its dual has the Radon-Nikod\'ym property  
iff the dual of every separable linear subspace of $V$ is separable.
Reflexive spaces and spaces of the type $c_0(\Gamma)$ are Asplund,
and every Asplund space is Rosenthal.) 
We denote by RN the class of {\em  Radon--Nikod\'ym dynamical systems} (see \cite{G-M}). 
A (metrizable) dynamical system is 
%060521 
said to be 
 RN if %and only if 
 it is representable on a (resp., separable) Asplund Banach space. 

Any RN system $X$ is \textit{hereditarily non-sensitive} (HNS) and any HNS system is tame.  
A metric system is RN iff it is 
%060521 adding HNS and removing HAE 
HNS.  
% \textit{hereditarily almost equicontinuous} (HAE), \cite{G-M}.
It was shown in \cite{G-M-U} that a metric system $(X,T)$ is RN iff its enveloping semigroup $E(X,T)$ is metrizable.

%060521
Any continuous topological group action on a dendrite is tame, \cite{GM-dendr}. Circularly (in particular, linearly) ordered dynamical $T$-systems are RN, hence tame, \cite{GM-c}.

\sk 
Yet another characterization of tameness, 
of combinatorial nature via the notion of independence tuples, is due to Kerr and Li \cite{K-L}. 
For more information and references about tame systems see \cite{GM-tLN,GM-survey}. 

\sk 

%060521
For 
%zzz
the
definitions of WAP and AP see Remark \ref{r:determ}. 
For the definition and basic properties of HNS we refer to \cite{G-M, G-M-U}. 
The following  inclusions hold in general 
%(not only for metrizable systems) 
$$
\mathrm{AP} \subset \WAP \subset \mathrm{HNS} \subset \Tame.
$$

\sk

\begin{prop} \label{r:incl} 
	For metrizable systems
	we have the following inclusions
	% for metrizable systems: %(which, in general, are strict, as we will see below).  
	$$
	\mathrm{RN}=\mathrm{HNS} \subset \mathrm{Tame}_\mathbf{2} \subset \mathrm{Tame}_\mathbf{1} \subset \Tame.
	$$
\end{prop} 

\begin{proof}	
The dynamical BFT dichotomy (Theorem  \ref{f:DynBFT}) implies that tame$_\mathbf{1}$ and 
tame$_{\bf 2}$ systems are tame 
(hence, $E(X,T)$ is a Rosenthal compactum in both cases). 
In fact, a tame$_\mathbf{1}$ system is tame because a first countable space is Fr\'{e}chet-Urysohn. 
%A  tame$_2$ system is tame because $\beta\N$ is not hereditarily separable. 
Finally, the above mentioned result of R. Pol 
%1411
(Theorem \ref{Pol}) implies that 
$\mathrm{Tame}_\mathbf{2} \subset  \mathrm{Tame}_\mathbf{1}$.  	
\end{proof}

\sk
%060521  Maybe to move this thm to other place ??
%3010
When a metric system %$(X,T)$ 
is tame$_\mathbf{2}$ we have the following detailed information.
This is a version of a 
%1111
theorem from \cite[Proposition 15.1]{G-M},
which in turn is based on results of S. Todor\u{c}evi\'{c}, (\cite[Theorem 3]{T}):

%1111
\begin{thm}\label{Todor-cor}
Let $X$ be a metric tame, point transitive $T$-system, where $T$ is 
%1211 a uniformly Lindel\"{o}f topological group.
an arbitrary topological group. 

Then 
\begin{enumerate}
\item
Either $E(X,T)$ contains an uncountable discrete subspace 
%0711
%1411 bold
(i.e. it is not tame$_\mathbf{2}$) or,
\item
There exists a metric tame system $(Z,T)$ and a factor map $Z \to X$ such that
$E(X,T) = E(Z,T)$ and such that, with $\pi \colon E(X,T) \to Z$, the evaluation map
$p \mapsto pz_0$, where $z_0 \in Z$ is any transitive point, we have
$$
|\pi^{-1}(z)| \leq 2,\quad \forall z \in Z.
$$ 
\end{enumerate}
\end{thm}

\begin{proof}
Suppose that $E(X,T)$ does not contain an uncountable discrete subspace.
Then there exists, by \cite[Theorem 15.1]{G-M}
a factor map $\phi \colon E(X,T) \to Y$ for some metric system $(Y,T)$
such that $|\phi^{-1}(y)| \leq 2$ for every $y \in Y$.
%1111
Let $x_0 \in X$ and $y_0 \in Y$ be transitive points and set $z_0 =(x_0, y_0)$ and 
$Z = \overline{T(x_0, y_0)}$. Then the system $(Z,T)$ is metric and tame with
$E(X,T) \to Z \to X$, such that
$E(Z,T) = E(X,T)$, 
and the evaluation map $\pi \colon p \mapsto p(x_0, y_0)$
from $E(X,T)  \to Z$ to is at most $2$ to $1$.
\end{proof}

\sk

\begin{remark}
Dynamical systems with the group $\Z$ of integers as the acting group are sometimes called {\em cascades}.
Often when dealing with cascades we change
 our notation and denote our system as $(X,T)$, 
where here the letter  $T$ denotes
the homeomorphism of $X$ which corresponds to $1 \in \Z$.
\end{remark}

\sk 
 %060521 implementing referee's suggestion 
\begin{remark} \label{r:LOCM} \ 
	\begin{enumerate}
		\item Note that every linearly ordered compact separable space $X$ is homeomorphic to a special linearly ordered space $K_A$ which can be obtained using 
%zzz
a splitting points construction. More precisely, by a result of Ostaszewski (see \cite{Ost} and its reformulation \cite[Result 1.1]{Mar}) for $X$ there exist: a closed subset $K \subset [0,1]$ and a subset $A \subset K$ such that $K_A=(K \times \{0\} \cup (A \times \{1\}))$ is endowed with the corresponding lexicographic order inherited from $K \times \{0,1\}$. Every splitting point is \textit{singular} (see definition in Section \ref{s:Order}). Hence, Lemma \ref{l:sing} implies that $K_A$ is metrizable if and only $A$ is countable. 
	
		\item 
		%zzz
		The splitting point construction can be generalized in several directions. Among others for \textit{circularly ordered} compact spaces. In \cite{GM-c,GM-int} we use such circularly ordered versions of Ostaszewski type spaces $K_A$ (which, as in \cite{GM-c}, 
%zzz
% sometimes 
we denote by $\Split(K; A)$) 
%2005 by splitting points of $A\subset K$ in the circularly ordered space $K$. 
the space which we get after splitting points of $A\subset K$ in the circularly ordered space $K$. 
%%
%zzz
We use such examples also in the present paper. In particular, in Proposition \ref{p:AlmostAut} we have the space $\T_A$, with $A:=\{n \alpha: n \in \Z\}$ and the double circle $\T_{\T}$ (see also Corollary \ref{cor:St}).
	\end{enumerate}	
\end{remark}

%060521 I suggest to discuss shortly the main results in the introduction 
\sk
\subsection*{Some 
%zzz
%results 
highlights
of the present work}

\begin{enumerate}
	\item In Theorem \ref{t:1-1} we give a sufficient condition for
	%zzz
the tame$_\mathbf{1}$ property for almost one-to-one extensions $X \to Y$ of 
%zzz
a
tame$_\mathbf{1}$ system
%zzz
$Y$. 
This implies, in particular, that $X$ is tame$_\mathbf{1}$ for the following symbolic systems: (a) Tribonacci 3-letter substitution system; (b) Arnoux-Rouzy substitutions; (c) Brun substitution; (d) Jacobi-Perron substitution.
\item 	
By Theorem \ref{t:Floyd} the Floyd minimal set 
%AA$_c$ but not AA$_{cc}$,  
is tame$_\mathbf{1}$.

	\item By Corollary \ref{grom} The action of a hyperbolic group $\Ga$ on its \textit{Gromov boundary} $\partial \Ga$ is tame but not tame$_\mathbf{1}$. 
	
	\item Proposition \ref{p:AlmostAut} shows that there exists a tame$_\mathbf{1}$ almost automorphic $\Z$-system which is not tame$_\mathbf{2}$. 
	
	\item By Theorem \ref{t:LinAreTame} 
	every linear action $\GL(n,\R) \times \R^n \to \R^n$ %on the Euclidean space $\R^n$
	is tame but, for $n \geq 2$ it is not tame$_\mathbf{1}$.   % for every subgroup $G \subset GL(n,\R)$. 
	\item 
As we 
%zzz
have shown in \cite{GM-c}
%already mentioned, 
circularly ordered $T$-dynamical systems are tame. 
This implies that many Sturmian like symbolic dynamical systems are tame. 
Now we can say more: 
Here we show that Sturmian like systems are tame$_\mathbf{2}$, Corollary \ref{cor:St}. 
%Moreover, any
And every linearly ordered metric system is tame$_\mathbf{1}$, Theorem \ref{t:LOT1}.
\item  
%zzz
%However, by 
By Proposition \ref{ex:T},  
the
circularly ordered system $(\T,H_+(\T))$ is tame but not tame$_\mathbf{1}$, where $H_+(\T)$ 
%be
is 
the Polish topological group of all c-order preserving homeomorphisms of the circle $\T$. Note also that 
%zzz
the
``circular analog of Helly's space" $M_+(\T,\T)$ 
(which is a separable Rosenthal compactum) is not first countable 
(see Remark \ref{r:c-Helly}).
%	\item 
%	\item 
\end{enumerate}

\br 
We thank Gabriel Debs, Jan van Mill, and Roman Pol for helpful (e-mail) discussions.
%zzz
We thank the anonymous referee for his many helpful suggestions.

\br
\section{$G_\del$-points in enveloping semigroups $E(X,T)$}
\sk

%2510 
%1811
We recall the following well known result.
%(see, for example, \cite[p. 378]{Gli}). 

\begin{lem} \label{Gdel} \cite[3.1.F]{Eng} 
Let $X$ be a compact space. A point $x_0 \in X$ is  a $G_\del$ point if and only if 
there is a countable basis for the topology at $x_0$.
\end{lem}
%2510 I think better to remove the proof
%1811
%060521 removing the proof (referee's suggestion) 
%\begin{proof}
%Suppose $\{V_i\}_{i \in \N}$ is a basis for the topology at $x_0$. Then clearly 
%$\{x_0\} = \bigcap_{i \in \N} V_i$ and $x_0$ is a $G_\del$ point.
%Conversely, suppose  $x_0$ is a $G_\del$ point and $\{x_0\} = \bigcap_{i \in \N} V_i$
%for a sequence $\{V_i\}_{i \in \N}$ of open subsets $V_i \subset X$.
%As $X$ is compact Hausdorff there is also a sequence of closed subsets $\{D_i\}_{i\in \N}$
%with $x_0 \in \intt{D_i}$ for each $i \in \N$, such that $\{x_0\} = \bigcap_{i \in \N} D_i$.
%We will show that the collection $\{U_i\}_{i \in \N}$, where $U_i = \intt{D_i}$,
%is a subbase for the topology at $x_0$.
%So let $U^n = \bigcap_{i =1}^n U_i$ and suppose, in order to get a contradiction, that 
%$\{U^n\}_{n \in \N}$ is not a basis for the topology at $x_0$.
%Then there is an open neighborhood $V$ of $x_0$ such that for each $n \in \N$ we can find a point 
%$x_n \in U^n \setminus V$. Let $x_{n_i} \to x$ be a convergent subnet.
%Then $x \in \bigcap_{n \in \N} U^n \subset \bigcap_{n \in \N} D_n = \{x_0\}$, 
%hence $x = x_0$. However we then have that eventually 
%$x_{n_i}$ is in $V$, which is a contradiction.
%\end{proof}

%2510 As in  \cite[Proposition 2]{Bour} we have:
The following useful result is a generalization of \cite[Proposition 2]{Bour}.

%3010
\begin{prop}\label{countable}
	Let $X$ be a set, $(Y,d)$  a metric space, 
	and $E \subset Y^X$ a compact subspace in the pointwise topology. 
	The following conditions are equivalent:
	\begin{enumerate}
		\item a point $p \in E$ admits a countable local basis in $E$;  
		\item there is a countable set $C \subset X$ which determines $p$
		%e3110 ADDING "in E"
		in $E$,  
		meaning that for 
		%060521 all $q \in E$, $q(c) =p(c) \ \forall c \in C$ implies $q(x)=p(x) \ \forall x \in X$.
		any given $q \in E$, the condition $q(c) =p(c)$ for all $c \in C$ implies that $q(x)=p(x)$ for every $x \in X$.
	\end{enumerate} 
\end{prop}

%\begin{lem}\label{countable}
%Let $(X,T)$ be a metric dynamical system with enveloping semigroup $E = E(X,T)$.
%Then a point $p \in E(X,T)$ admits a countable basis for its topology iff
%there is a countable set $C \subset X$ such that for all $q \in E(X,T)$,
%$qc =pc, \ \forall c \in C$ implies that $q =p$.
%\end{lem}

\begin{proof}
	We note that the following collection of subsets  forms 
	a basis for the topology at a point $p \in E$:
	\begin{equation}\label{basis}
	V(x_1, x_2, \dots, x_n; \ep) =
	\{q \in E : d(px_i, qx_i) < \ep; \  1\leq i \leq n\},
	\end{equation}
	with $\{x_1, x_2, \dots, x_n\}$ a finite subset of $X$ and $\ep >0$.

	Assuming that $p$ admits a countable basis for its topology, there is then also a basis
	for this topology whose elements of the form (\ref{basis}), and the union 
	of all those finite sets that appear in this representation clearly forms a countable set 
	$C \subset X$ which determines $p$.
	
	Conversely, when $p$ is determined by its values on a countable subset $C \subset X$,
	it is clearly a $G_\del$ point of $E$ and thus, by Lemma \ref{Gdel} it admits a countable base for its topology.
	\end{proof}

\begin{remark} \label{r:determ} \ 
	\begin{enumerate}
		\item Recall that a $T$-system $X$ is WAP (weakly almost periodic) if and only if every 
		$p \in E(X,T)$, as a function $p \colon X \to X$, is continuous 
		%3010
		\cite{E-N}. 
		Clearly then for a metric WAP system
		each countable  dense subset $D \subset X$ determines every $p \in E(X,T)$. 
		\item A system $Y$ is AP (= equicontinuous) if and only every $p \in E(X)$ is a homeomorphism of $X$. In this case $E(X,T)$ is a compact topological group. 
		If, in addition, the action of $T$ on $X$ is
%3010		% topologically 
		 point transitive (i.e., has a dense orbit) then $X$ is a coset space $E(X,T) / H$ for some closed subgroup $H$ of $E(X,T)$. 
When $T$ is abelian, $H$ is trivial and $X$ is identified with the topological compact group 
$E(X,T)$. In this case each element $p \in E(X,T)$ is completely determined 
by its value at any single point $x$ of $X$. 
	\end{enumerate}
	
\end{remark}

\begin{remark} 
Note that if one can choose a countable subset $C \subset X$, as in Proposition \ref{countable},
which is independent of $p$, then the map $p \mapsto (px)_{x \in C} \in X^C$,
is a 
%e3110 homeomorphism, 
topological embedding, 
and it follows that $E(X)$ is metrizable. Thus, in view of \cite{G-M-U},
the existence of such a set is a 
%necessary \footnote{\textcolor{blue}{why "necessary" ?}}  and 
sufficient condition for a metric $(X,T)$ to be a HNS system.
%3010
(It is also necessary, because cylindrical sets in $E(X) \subset X^X$
form a basis for the topology, so that when $E(X)$ is a metric space,
this basis admits a countable subfamily which is also a basis.)
\end{remark}

\sk   
By a theorem of Bourgain \cite{Bour}, in every Rosenthal compactum $K$ the set of $G_\del$-points is dense.
As remarked by Debs,
since any non empty $G_\del$ subset of $K$ contains a non empty compact 
$G_\del$ subset, 
%({\color{red}{is that easy to see ? he gives no reference to this claim}})
it follows from Bourgain's result that the set of all $G_\del$-points of a 
Rosenthal compactum is actually non meager. 
[In fact, let $D \subset X$ be the collection of $G_\del$-points in $X$.
If $D$ is meager it is contained in a union $\bigcup_{n \in \N} K_n$ where each 
$K_n$ is closed with empty interior. Now 
$$
X \setminus \bigcup_{n \in \N} K_n
= \bigcap_{n \in \N} K_n^c
$$ 
is a $G_\del$ subset of $X$, hence contains a compact
$G_\del$ set which, in turn, contains a $G_\del$ point by the proof of Bourgain's theorem].

Debs remarks in \cite{Debs} however that the question whether  the set of all $G_\del$-points of a Rosenthal 
compactum $K$ is comeager in $K$ is still open.

\sk

\begin{question} \label{pr:G-delta} 
	%060521 (referee's suggestion) 
	For a metric tame system $(X,T)$
is the set of all $G_\del$-points of $E(X,T)$ 
%residual ? 
comeager ?
%060521 $E(X,T)$, for $(X,T)$ a metric tame system, residual in $E(X,T)$?
\end{question}

%zzz Moved down
%\begin{question} \label{pr:RosComp=E} 
%Among the class of separable Rosenthal compacta, which members are homeomorphic to
%$E(X,T)$ 
%(or $\Adh(X,T)$, see Section
%%3010 
%%\ref{sec-asy})
%\ref{wr})
%for some system $(X,T)$ (with $(X,T)$ metric, minimal, $T$ abelian, $T = \Z$ etc.) ?
%\end{question}

\sk 
In order to formulate a related question we need some notations. 
%{\bf A question regarding the set of $G_\del$ points}

\sk 

%\nt {\bf Notations:}
Let $I$ be the unit interval, $\Om = I^I, \Sigma = \{0,1\}^I$. Let
$$
\Sigma_c = \{\sigma \in \Sigma : |\supp(\sigma)| \leq \aleph_0\}.
$$
For $p \in \Om$ and $C \subset I$ let 
$$
\Om_C(p) = \{q \in \Om : q \rest C = p \rest C\}.
$$  

\sk 

Now, suppose $A \subset \Om$ is a closed subset. Consider the sets:
\begin{gather*}
	\Gcal = \{(p, C) : \Om_C(p) \cap A = \{p\}\} \subset A \times \Sigma_c,\\
	G = \{p \in A : \exists C \in \Sigma_c, \ \Om_C(p) \cap A = \{p\}\} = \pi_1(\Gcal),
\end{gather*}
where 
%zzz
elements of $\Sigma_c$ are identified with their supports and
$\pi_1$ is the projection on $\Om$. 
So $G$ is a kind of an `analytic' set. 

\sk

\nt {\bf Question}:
Is  $G$ a Baire set;
i.e. of the form $U \tri M$ for $U \subset \Om$ open and $M \subset \Om$ meager ?

\sk 
In a private communication Jan van Mill have shown that in general the answer to this
question may be negative.
However, in our case we have further information concerning the set $A$ that might be relevant.
The compact set $A \subset \Omega$ is 
a separable Rosenthal compactum (so in particular a Fr\'{e}chet-Urysohn space). 
Thus by Bourgain's theorem it has a dense (and non-meager) subset of $G_\delta$ points.

%3010
Note that if we know that the set of $G_\del$ points in a separable compactum $K$ is a Baire set,
then it is comeager.

\begin{question} \label{pr:RosComp=E} 
Among the class of separable Rosenthal compacta, which members are homeomorphic to
$E(X,T)$ 
(or $\Adh(X,T)$, see Section
%3010 
%\ref{sec-asy})
\ref{wr})
for some system $(X,T)$ (with $(X,T)$ metric, minimal, $T$ abelian, $T = \Z$ etc.) ?
\end{question}

\br 
\section{When is an almost automorphic system tame ?}
\label{s:AlmostAut} 
\sk 

%3010
Almost automorphic dynamical system were first studied by Veech in \cite{V65}.
Since then they play a central role in various aspects of the theory of dynamical systems.

\begin{defin} \label{d:str1-1} Let $\pi \colon (X,T) \to (Y,T)$ be a factor map of 
metric minimal dynamical systems. 
	\begin{enumerate}
		\item $\pi$ is said to be an \emph{almost one-to-one extension} if 
		the set 
		%0605 replacing =x   by   =\{x\}
		$$
		X_0:=\{x \in X: \ \pi^{-1}(\pi(x))= \{x\}\}
		$$ 
		is a dense $G_\del$ subset of $X$.  
		When $\pi$ is an almost one-to-one extension we let
		$$
		Y_0 = \pi(X_0) = \{y \in Y : |\pi^{-1}(y)| = 1\}.
		$$
		When the set $Y \setminus Y_0$ is countable we say that $\pi$ is of 
		%0312   maybe better to write "type (c)"  (  instead of "type c"  )  
		{\em type c},
		and we say that $\pi$ is of {\em type cc} when $X \setminus X_0$ is countable.
		%2810
		\footnote{In \cite{GM-tLN} we used the term `strongly almost one-to-one' for the
		type cc property.}
		\item  
		%060521 Changing the formulation of (2). Splitting it by adding (3) 
		%If $\pi$ is an almost one-to-one extension and the system $(Y,T)$ is equicontinuous we say that the system $(X,T)$ is {\em almost automorphic}.
		$(X,T)$ 
		is \textit{almost automorphic} if there exists an almost one-to-one extension $\pi \colon (X,T) \to (Y,T)$ with an equicontinuous system $(Y,T)$; 
		we denote by AA the class of almost automorphic systems.
	\item 	
We say that an almost automorphic system $(X,T)$ is:
\begin{itemize}
	\item [(a)] of {\em type} AA$_c$ when 
	the corresponding set $Y \setminus Y_0$ is countable; 
	\item [(b)] of {\em type} AA$_{cc}$ when
	$X \setminus X_0$ is countable.
\end{itemize} 
\end{enumerate}	
Thus, we have
$$
{\rm AA}_{cc} \subset {\rm AA}_c \subset {\rm AA}.
$$ 
\end{defin}

\br  
%0111
As was shown in \cite[Corollary 5.4.(2)] {G-18},  a minimal metric tame system $(X,T)$ 
which admits an invariant
measure is almost automorphic and moreover, the invariant measure is unique and the 
almost one-to-one map $\pi \colon X \to Y$, from $X$ onto its maximal equicontinuous factor,
is an {\em isomorphic extension}, i.e. $\pi$
is a measure theoretical isomorphism when $Y$ is equipped with its Haar measure.
It was shown recently in 
%that 
%2111
\cite[Corollary 3.7]{FGJO} that a tame AA system is necessarily regular
(i.e. the unique invariant measure is supported on the set $X_0$ of singleton fibers).
%The converse does not hold; e.g. 
Of course there are many examples of AA systems which are not tame.
For example there are  almost automorphic systems having more than one invariant measure,
or having positive entropy, and
%and, as was shown recently in \cite[Corollary 3.7]{FGJO} a tame AA system is necessarily regular.
in \cite[Theorem 3.11, Corollary 3.13]{F-K} there are examples of regular AA systems 
%with $\pi$ an isomorphic extension
which are not tame. 
 
\sk

%060521  it should be more clear in the proof how we use the assumption "abelian" 
\begin{prop}\label{str-1-1}
Let $(X,T)$ be an AA$_{cc}$ system with $T$ abelian.
%is a strongly almost one-to-one extension $\pi : X \to Y$ of its Kronecker factor $(Y,T)$. 
Then $E(X,T)$ is first countable (hence $(X,T)$ is tame$_\mathbf{1}$).
\end{prop}

\begin{proof}
Note first that for a minimal equicontinuous system $(Y,T)$
%zzz
with $T$ abelian, an element $p \in E(Y,T)$ is
completely determined by its value at any single point $y$ of $Y$ 
%060521 
(Remark \ref{r:determ} (2)). 
Now let $\pi \colon X \to Y$ be the largest equicontinuous factor of $(X,T)$ and
let $X_0 = \{x \in X : \pi^{-1}(\pi(x)) = \{x\}\}$. By assumption the set $C = X \setminus X_0$ is countable.
%Pick a countable dense subset $C \subset X$.
Given $p \in E(X,T)$, consider the set $C_p = \{ x \in X : px \in C\}$.
This is at most countable, and taking into account the above remark,
if we know the restriction of $p \in E(X,T)$ to $C_p$, we know $p$.
Now apply Proposition \ref{countable}.
\end{proof}

\sk 
A more general statement is as follows:

%Here every fiber $\pi^{-1}(y)$ is at most countable \ $\forall y \in Y$.  

%2808 
\begin{thm} \label{t:1-1}  
Let $\pi \colon  (X,T) \to (Y,T)$ be an almost one-to-one extension of type 
%2111
cc,
and let $\pi_* \colon E(X) \to E(Y)$ be the induced surjective homomorphism.
Suppose $p \in E(X)$ satisfies:
\begin{itemize}
	\item [(a)] $E(Y)$ is first countable at $\pi_*(p)$;  
	\item [(b)] $\pi_*(p)^{-1}(y)$ is at most countable for every $y \in Y$.
	%0312 I think (b) cannot be derived from (cc)
\end{itemize}

Then $E(X,T)$ is first countable at $p$. Furthermore, if the conditions above (a), (b) 
are true for every $p \in E(X)$ then $X$ is tame$_\mathbf{1}$. 
\end{thm}
\begin{proof} In order to prove that $E(X)$ is first countable at $p$, 
according to Proposition \ref{countable}, 
we have to show that there exists a countable subset $C(p)$ 
which ``determines" $p$. That is, $p'(x)=p(x)$ for every 
$x \in C(p)$ and $p' \in E(X,T)$ implies $p=p'$.  
	
	For our $p \in E(X)$ choose first 
	a countable subset $C(\pi_{*}(p))$ of $Y$ which ``determines" $\pi_{*}(p)$ in $E(Y)$ 
	(such a set exists by Proposition \ref{countable} because $E(Y)$ is first countable at $\pi_*(p)$). Fix a countable subset $A$ of 
%2111	
	%$E(X)$ 
	$X$
	such that $\pi(A)=C(\pi_*(p))$. 
	
	By Definition \ref{d:str1-1} the set $X \setminus X_0$ is countable, where 
	$X_0: =\{x \in X| \ \pi^{-1}(\pi(x))=\{x\}\}$. Set $Y_0:=\pi(X_0)$.  
	Define now 
	$$B:=\pi^{-1} (\pi_*(p)^{-1}(Y \setminus Y_0))).$$ 
	Since $Y \setminus Y_0$ is countable 
	then $\pi_*(p)^{-1}(Y \setminus Y_0))$ is countable (use (b)). 
	%0312 It follows (see Definition \ref{d:str1-1}) that $B$ is also countable.
	Since $\pi$ is of type cc (Definition \ref{d:str1-1}(1)) it follows that the preimage $\pi^{-1}(y)$ is countable for every $y \in Y$. Hence, $B$ is also continuous. 
	Now we claim that the countable set 
	$$C(p) := A \cup B$$
	determines $p$. It is enough to verify the following 
	
	\sk 
%	\begin{claim}
   \nt \textbf{Claim.} 
	\textit{Let $p', p \in E(X)$. 
	\begin{itemize}
		\item [(i)] If $p'(a)=p(a) \ \forall a \in A$. Then $\pi_*(p')=\pi_*(p)$. 
		\item [(ii)] Assume that $\pi_*(p')=\pi_*(p)$. If $p'(x)=p(x) \ \forall x \in B$ then $p'=p$. 
	\end{itemize} }
%	\end{claim}
	\sk 
	\nt \textit{Proof of} (i): note that 
	$\pi p'(a) = \pi p(a)  \ \forall a \in A$. Therefore, $\pi_*(p')(\pi(a))=\pi_*(p)(\pi(a))$ \ $\forall a \in A$. So, 
	$\pi_*(p')(y)=\pi_*(p)(y) \ \forall y \in C(\pi_*(p))$. This means that $\pi_*(p')=\pi_*(p)$. 
	
	\sk 
	\nt \textit{Proof of }(ii): it is enough to show that for every $x \notin B$ we have 
	$p'(x)=p(x)$. 
	
	So let $x \notin B$. Then 
	$\pi_*(p) (\pi(x)) \in Y_0$. By definition of $X_0$ there exists unique element $z \in X$ such that $\pi(z) = \pi_*(p) (\pi(x))$. Since $\pi p(x)=\pi_*(p) \pi(x)$ we have  $p(x)=z$. Moreover, since $\pi_*(p')=\pi_*(p)$ we have also  $p'(x) = z$. So, $p'=p$. %This proves the Claim.  
\end{proof}

\sk  
%2808  
\begin{remark} \label{r:corollaries} \ 
	\begin{enumerate}
		\item The dynamical system $(X,T)$ from 
		Theorem \ref{t:1-1}, in particular, is tame (being tame$_{\bf 1}$, see Proposition \ref{r:incl}). This 
		strengthens \cite[Theorem 6.16]{GM-tLN}.  
		Similarly, Proposition \ref{str-1-1}, in turn, strengthens a result of Huang \cite{H}.  
		
		\item Proposition \ref{str-1-1} together with results of Jolivet \cite[Theorem 3.1.1]{diss:Jolivet}, imply that $E(X)$ is first countable (i.e., $X$ is tame$_\mathbf{1}$) for the following symbolic systems: (a) Tribonacci 3-letter substitution system; (b) Arnoux-Rouzy substitutions; (c) Brun substitution; (d) Jacobi-Perron substitution. 
		%0711 
		%0611 Of course this leads to a natural question if these systems are even tame$_2$ or not ? 
		%Implicit construction of the enveloping semigroups in these cases (say, for Tribonacci and Floyd) would have its own interest. 
	\end{enumerate}
	
\end{remark}

\br 
%27
A similar argument will show that the Floyd minimal set \cite{F},
which is AA$_{c}$ but not AA$_{cc}$, is tame$_{\bf 1}$. 
For more details on this dynamical system
see \cite{Au} and \cite{H-J}.

\begin{thm} \label{t:Floyd} 
	The Floyd minimal set, which is
	AA$_c$ but not AA$_{cc}$,  is tame$_\mathbf{1}$. %with a first countable enveloping semigroup.
	%0611 Eli, can you see if Floyd system is not tame$_2$ ? If YES, then Proposition \ref{p:AlmostAut} will be more interesting
\end{thm}  
\begin{proof}
	Let $(X,T)$ be the Floyd minimal system, with $\pi : X \to Y$ the almost one-to-one
	factor map onto its largest equicontinuous factor $Y$, which is the $3$-adic adding machine.
	In Figure 1 we recall an idea constructing the  homeomorphism which generates the Floyd's cascade. 
	\begin{figure}[h]
		\begin{center} \label{F1}
			\scalebox{0.7}{\includegraphics{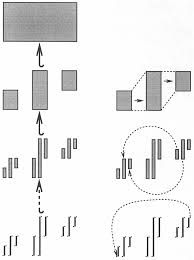}}
			\caption{Floyd system (from JSTOR)}
		\end{center}
	\end{figure}  
	As in \cite{Au} we consider $X$ as a subset of the square $[0,1] \times [0,1]$ in $\R^2$,
	$Y$ as a subset of the unit interval $[0,1]$, and the map $\pi$ as the projection on the first coordinate. 
	We let
	$$
	X_0 = \{x \in X : \pi^{-1}(\pi(x)) = \{x\}\}, \quad Y_0 = \pi(X_0).
	$$
	We know that the set $Y_1 =Y \setminus Y_0$ is countable and that for
	each $y \in Y_1$ the set $A_y = \pi^{-1}(y)$ is a non-degenerate vertical line segment over the point $y$.
	It is not hard to see that for each $p \in E(X,T)$, and such line segment $A_y$, the restriction 
	of $p$ to $A_y$ is an affine map of $A_y$ onto a closed subinterval of $A_{py}$,
	if $py \in Y_1$, and into the singleton $\pi^{-1}(py)$ if $py \in Y_0$.
	As a consequence we see that for $y \in Y_1$, the map
	$p \rest A_y$ is determined by the values of $p$ at any two distinct points of $A_y$.
	Choose a set $C \subset X_1 = X \setminus X_0$ such that for each $y \in Y_1$ we have that
	%$|C \cap A_y| =2$. 
	$C \cap A_y$ is a countable set which is dense in $A_y$. Of coarse $C$ is a countable set.
	
	Now given $p \in E(X,T)$, let $C_p = \{x \in X : px \in  C\}$. 
	We observe that, because $p$ restricted to a non-degenerate vertical segment 
	which is mapped by $p$ onto a non-degenerate vertical segment, is an affine map (hence one-to-one)
	$C_p$ is a countable set.
	We will show that
	$p$ is determined by its values on $C_p$; i.e. if $q \in E(X,T)$ is such that $p \rest C_p = q \rest C_p$
	then $p =q$.
%2808  
	We first note that $E(Y,T)=Y$ is an abelian group and by 
	Remark \ref{r:determ}.2 any single element of $E(Y)$
	 is completely determined by its value at any single point $y$ of $Y$. 
	 So, $p$ and $q$ have the same image in $E(Y,T)$. 
	It then follows that $px = qx$ for all $x \in X_0$. So we now assume that $x \in X_1 = X \setminus X_0$.
	
	{\bf Case 1:} $px \not\in C$.
	Then $\pi^{-1}(\pi(px)) = \{px\}$ and, since $\pi(px) = \pi(qx)$, we have $qx = px$.
	
	{\bf Case 2:}  $px \in C$; i.e. $x \in C_p$, hence $y = \pi(px) \in Y_1$ and $A_y$ is a non-degenerate interval.
	Let $z = \pi(x)$, which by assumption is in $Y_1$. Thus $A_z$ is a non-degenerate interval containing $x$.
	
	{\bf Case 2a:}  $pA_z = \{px\} \subset A_y$ is a singleton. It then follows that also
	$qA_z = \{px\}$, hence $qx = px$.
	
	{\bf Case 2b:} $pA_z \subset A_y$ is a non-degenerate interval.
	In this case $|pA_z \cap C| = \infty$ and there are distinct points $x_1, x_2 \in A_z$
	so that $px_1$ and $px_2$ are distinct points of $pA_z \cap C$. In particular
	$x_1, x_2 \in C_p$ and by assumption $qx_1 = px_1$ and $qx_2 = px_2$.
	As both $p \rest A_z$ and $q \rest A_z$ are affine maps,
	%2808
	%\footnote{ \textcolor{blue}{I guess it should be mentioned that these affine maps are defined on $\R^2$} }
	 it follows that
	$p \rest A_z = q \rest A_z$ and, in particular $qx = px$.
	
	Thus we have shown that indeed $p=q$, and the countable set $C_p$ determines $p$.
	Applying Proposition \ref{countable} we deduce that $E(X,T)$ is first countable,
	and hence also that $(X,T)$ is tame$_\mathbf{1}$.
\end{proof}

\sk 

\begin{ex}
	In a beautiful work  \cite{Auj}, Aujogue gives a complete description of the enveloping semigroup
	of a large class of cut and project $\R^d$ systems called {\em  almost canonical model sets} 
	%2808 
	(which play a major role in mathematical models of quasicristals). 
	These are almost automorphic $\R^d$ actions obtained by the cut and project scheme,
	from a lattice $\Sigma$ in  $\R^n \times \R^d$ with a window $W \subset \R^n$,
	which is a polytope satisfying certain conditions.
	In particular  Aujogue shows that these systems are tame with a first countable enveloping semigroup. 
	%0312  
	See also \cite{ABKL}. 
\end{ex}

%27
The interested reader will find many more examples of tame
and non-tame almost automorphic systems in the recent works
\cite{F-K}, \cite{F-K-Y}.

%\noindent\hrulefill

%\vskip 0.5cm 
\br 
\section{Asymptotic extensions}\label{sec-asy}
\sk

Recall that the {\em adherence semigroup} $\Adh(X,T)$ of a dynamical system $(X,T)$ 
is the set of accumulation points of $T$ in $E(X,T)$. It is the
semigroup $E(X,T) \setminus T$ 
%2808 when 
iff
%0711
%$T$ consists of isolated points in $E(X,T)$
$T$ forms an open dense subset of $E(X,T)$,
iff $(X,T)$ is not weakly rigid 
%0611 why "adherence semigroup" is always semigroup for every acting semigroup ? Do you have a reference ? In \cite{BGKM} the actions are semi-cascades. 
 (see \cite{BGKM} and Section \ref{wr} below).
An extension $\pi \colon (X,T) \to (Y,T)$ of minimal systems is called {\em asymptotic} if 
%060521  Please fix it (read referee's suggestions n. (10) )
$R_\pi \subset {\rm ASYM}$, where 
%zzz
%the latter is the relation
$$
R_\pi =  \{(x, x') \in X \times X : \pi(x) = \pi(x')\},
$$
and 
$$
{\rm ASYM} = \{(x, x') \in X \times X : px = px', \ \forall p \in \Adh(X,T)\}.
$$
By \cite[Theorem 6.29]{AAG} an asymptotic extension is irreducible and, when $X$ is metrizable,
it is an almost one-to-one extension.
If in addition the system $(Y,T)$ is equicontinuous, then $(X,T)$ is almost automorphic.
%0111
%By \cite[Corollary 5.4 (2)]{G-18} a minimal metric tame dynamical system which
%admits a (necessarily unique) invariant probability measure is AA; 
%i.e. it is an almost one-to-one extension of its maximal equicontinuous factor.
%3010
Note that the Floyd system is not an asymptotic extension of its maximal
equicontinuous factor (it has, for example, infinitely many fibers of length $1/2$).

\begin{thm}\label{asym}
Let $(X,T)$ be a minimal metric system, with $T$ an abelian group.
Suppose further that it is AA$_c$, and that the extension $\pi \colon X \to Y$,
where $Y$ is the largest equicontinuous factor of $X$, is asymptotic. 
%admitting an invariant probability measure.
%Let $\pi : (X,T) \to (Y,T)$ be the largest equicontinuous factor of $(X,T)$, so that
%$\pi$ is an almost one-to-one extension.
%Suppose, moreover that $\pi$ is an asymptotic extension and that the 
%set $Y_m = \{y \in Y : |\pi^{-1}(y)| > 1\}$ is at most countable.
Then $X$ is tame$_\mathbf{1}$. %$E(X,T)$ is first countable.
\end{thm}

\begin{proof}
Let $\pi \colon (X,T) \to (Y,T)$ be the largest equicontinuous factor of $(X,T)$, so that
$\pi$ is an almost one-to-one extension.
By assumption the set
$$
Y_m = \{y \in Y : |\pi^{-1}(y)| > 1\}
$$ 
is at most countable.
Let 
\begin{gather*}
X_m = \pi^{-1}(Y_m),\\ 
%X_0 = \{x \in X : \pi^{-1}(\pi(x)) = \{x\}\} = X \setminus X_m,\\
%Y_0 = \pi(X_0)
%= \{ y \in Y : |\pi^{-1}(y)| =1\} = Y \setminus Y_m\\
X_0 = \{x \in X : \pi^{-1}(\pi(x)) = \{x\}\} = X \setminus X_m,  \ {\text{and}}\\
Y_0 = \pi(X_0) = \{ y \in Y : |\pi^{-1}(y)| =1\} = Y \setminus Y_m.
 \end{gather*}
If $X_0 = X$, the system $(X,T)$ is equicontinuous and $E(X,T) \cong (X,T)$ is metrizable.
So we now assume that $X_0 \subsetneq X$, in which case it is a dense $G_\del$ subset of $X$.
%Let $X_m \subset \pi^{-1}(Y_m)$ be a subset of $X$ such that $|X_m \cap \pi^{-1}(y)| = 2$
%for every $y \in Y_m$.
%Clearly $X_m$ is countable. 
We recall that the map $\pi$ induces a surjective (dynamical and semigroup) homomorphism 
$\pi_* \colon E((X,T) \to E(Y,T)$. In our case, as $(Y,T)$ is minimal, metric, equicontinuous and 
$T$ is abelian,
$E(Y,T) \cong (Y,T)$ is
%2808 
%\footnote{ \textcolor{blue}{the acting group is abelain ?}}
 a second countable compact abelian topological group.
We have $\pi(px) = \pi_*(p)\pi(x)\ (x \in X, p \in E(X,T))$. For simplicity we will dispense with the symbol $\pi_*$
and write this as $\pi(px) = p\pi(x)$. Note that then, considering $p$ as an element of 
the group $E(Y,T)$, it has an inverse and the notation $p^{-1}y$ for $y \in Y$ makes sense.
Now given $p \in E(X,T)$ choose 
$C(p) \subset \pi^{-1}(p^{-1}Y_m)$ such that 
\begin{gather}\label{12}
\begin{split}
| C(p) \cap   \pi^{-1}(y)| = 1,\quad  {\text {if}}\  y \in p^{-1}Y_m \cap Y_0\\
| C(p) \cap   \pi^{-1}(y)| = 2, \quad {\text {if}}\   y \in p^{-1}Y_m \setminus Y_0.
\end{split}
\end{gather}
We write $C(p) = C_1 \cup C_2$ where the sets $C_1, C_2$ are determined by the conditions in (\ref{12}).

Suppose now that $q \in E(X,T)$ is such that $q \rest C(p) = p \rest C(p)$;
we claim that $p=q$.
Let $x$ be any point in $X$ and let $y = \pi(x)$.
%2808 ???
%\footnote{ \textcolor{blue}{the following argument should be modified} } \textcolor{blue}{
	As $C(p)$ is nonempty our assumption implies that $p$ and $q$ define the same element of the group $E(Y,T)$. Thus if $px \in X_0$ then clearly $px = qx$. 
So we now assume that $px \in X_m$, and it then follows that $x \in \pi^{-1}(p^{-1}Y_m)$.
If $y \in p^{-1}Y_m \cap Y_0$ then $\pi^{-1}(y) = \{x\}$ and we have $qx = px$.
Otherwise, $y \in p^{-1}Y_m \setminus Y_0$ and $\pi^{-1}(y)$ is not a singleton.
We have to consider two cases.

In case $p \in \Adh(X,T)= E(X,T) \setminus T$, then the assumption that $\pi$ is an 
asymptotic extension implies that $p \pi^{-1}(y) = p(\pi^{-1}x)$ is a singleton, say $\{z\}$, and then 
$px = z = px_1 = px_2$, where $\{x_1, x_2\} = C_2 \cap \pi^{-1}(y)$, hence $z = px_1 = qx_1 = px_2 = qx_2$. 
It follows that also $q \in \Adh(X,T) = E(X,T) \setminus T$, hence $qx = px =z$.
Finally if $p \in T$ then $qx_1 = px_1 \not = px_2 = qx_2$, therefore also $q \in T$, whence 
$p=q$.
In view of Proposition \ref{countable} our proof is complete.
\end{proof}

\sk

Now the class of dynamical systems satisfying the conditions of Theorem \ref{asym} is very large.
It contains e.g. all the Sturmian like systems 
%3010
and many other semi-cocycle extensions (see subsection \ref{semi} below).
%but also many other cascades. 
We demonstrate this with the following construction.

\br

\begin{exa}
Here is a simple construction of a minimal cascade $(X,T)$ on a Cantor set
with a factor map $\pi \colon X \to Y$ such that 
%0607 $N(y): =\card \pi^{-1}(y) < \infty$ 
$N(y): =\card \pi^{-1}(y) < \omega$ 
for all $y\in Y$ but $N(\cdot)$ is not bounded.

Start with an adding machine (or any other minimal system on a Cantor set),
say $(Y,S)$. Choose a convergent sequence of distinct points
$\{y_n\}_{n=1}^\infty$ in $Y$, with $\lim_{n\to\infty} y_n =y_\infty \not\in \{y_n\}_{n \in \N}$,
such that the orbits $\Ocal_S(y_i), \ i=1,2, \dots,\infty$, are all distinct.
Construct by induction a sequence of disjoint clopen neighborhoods
$V_n \ni y_n$,\  $n< \infty$.
Choose for each $n$ a nesting sequence of clopen sets $B_{n,k}$, neighborhoods of
$y_n$, with $B_{n,1} \subset V_n$ and $\lim_{k\to\infty}{\diam} \ B_{n,k}=0$.
Next define a function $f \colon Y \to [0,1]$ as follows.
For a fixed $n<\infty$ let $f = \frac{j}{n^2}$ on $B_{n,k}\setminus B_{n,k+1}$
where $0 \le j \le n-1$ and $j\equiv k \pmod{n}$. On the rest of the space
put $f=0$. Clearly $f$ is continuous on the complement of $\{y_1,y_2,\dots\}$.
In particular $f$ is continuous at $y_\infty$.

Let $y_0 \in Y$ be a point of continuity of $f$ outside of $\{y_n\}_{n=1}^\infty \cup \{y_\infty\}$,
and let $X$ be the closure of the orbit of $(y_0,\{f(S^j y_0)\}_{-\infty < j < \infty})$
in $Y \times [0,1]^\Z$, where on $[0,1]^\Z$ the action is the shift $\sig$,
and $T$ is defined as the restriction to $X$ of the product $S \times \sig$.

It is not hard to verify that $\pi \colon X \to Y$, the projection on
the first coordinate, has the property ${\card}\pi^{-1}(S^jy_k)= k$, for $k=1,2,\dots$,
$j\in\Z$, and that ${\card}\pi^{-1}(y)= 1$ for any other point $y\in Y$.
As a point transitive almost 1-1 extension of $Y$, the system $(X,T)$ is minimal.

When we start with $(Y,S)$ which is an adding machine, it is easy to check,
that the system $(X,T)$ is tame and that the extension $\pi$ is asymptotic.
We thus obtain a metric tame system $(X,T)$ with $E(X,T)$ first countable.

Similar constructions will yield many other examples of this type with, say,
$|\pi^{-1}(y)| = \aleph_0$ or  $\pi^{-1}(y) = [0,1]$, for some $y \in Y$; or in fact, given any compact separable
space $\Om$, we can build such an example with $\pi^{-1}(y) = \Om$ for some $y \in Y$.
See \cite[Example 5.7]{Y-Z} for such constructions.
\end{exa}

\sk

%3110

\begin{remark}
In Section \ref{mi}, Theorem \ref{mli}, we will show that, e.g. when $T$ is abelian,
the condition AA$_c$ suffices to ensure that if $(X,T)$ is a minimal metric tame system,
then the unique minimal left ideal $M$ in $E(X,T)$ is a first countable Rosenthal compactum.
\end{remark}

%\noindent\hrulefill

%\vskip 0.5cm 

%-------------------------------------------

%27
\br
\subsection{Semi-cocycles}\label{semi}
\sk

Let $T$ be a discrete countable group.
Let $(Y,T)$ be a metric minimal infinite dynamical system and let $K$ be a compact space.
Let $C \subset Y$ be a nonempty finite or countable set. 
%$Ty \cap Ty' = \emptyset$ for every distinct points $y, y' \in C$.
Let 
$F \colon X \setminus C \to K$ be  a continuous map.
Such a function is called {\em a semi-cocycle} in 
\cite{D-D}, \cite{D}.
We refer to \cite{D-S} and the recent \cite{F-K} for the theory of 
semi-cocycles, whose roots can be traced to \cite{Fu}.

We further assume that for each $c \in C$ the 
function $F$ can not be extended continuously to $c$,
and that {\bf for every $c \in C$ the set $Tc \cap C$ is finite}. 

We fix a point $y_0 \in X \setminus C$ and set $f(t) = F(ty_0 )$, a function in $\Om = K^T$.
%We consider $F$ as a function on $\T = \R/\Z$.
Let $x_0 \in Y \times K^T$ be defined by
$$
x_0(t) = (y_0, F(ty_0)) = (y_0, f(t)), \quad t \in T.
$$
For $t \in T$ let $T_t \colon Y \times \Om  \to Y \times \Om$ be defined by
$$
T_t(y, \om) = (ty, S_t\om), 
$$
where for $\om \in \Om$ the shift homeomorphism $S_t \colon \Om \to \Om$ 
is defined by $S_t\om(s) = \om(st)$.
Let 
$$
X = \cls \Ocal_T(x_0) = \overline{\{T_t x_0 : t \in T\}} \subset  Y \times \Om.
$$
We let $\pi \colon X \to Y$ denote the projection on the $Y$ coordinate.
Clearly $\pi \colon (X,T) \to (Y, T)$ is a homomorphism of dynamical systems.
One can also easily see that $\pi$ is an almost one-to-one extension, with
$\pi^{-1}(\pi(x)) = \{x\}$ for every $x$ with $\pi(x) \not\in TC$.
In fact, for such $x$ and for a sequence $t_i$ such that $\lim t_i y_0 =\pi(x) = y \not \in TC$, we have:
$$
\lim_{i \to \infty} T_{t_i}x_0 = x, \quad \text{with} \quad 
x(t) = (y, F(ty)),\quad (t \in T).
$$
It then follows that the system $(X,T)$ is minimal and almost automorphic of type AA$_c$.

Next fix a point $c \in C$ and let $t_i$ be a sequence in $T$ such that
$x = \lim T_{t_i} x_0$ with $\pi(x) = c$. We then have
$$
x(t)  = \lim_{i \to \infty}( T_{t_i}x_0 )(t) =  \lim_{i \to \infty}(t_iy_0, f(tt_i)) =
 \lim_{i \to \infty} (t_iy_0, F(tt_iy_0)) = (c, k(t)) ,\quad (t \in T)
$$
with  $k(t) =  \lim_{i \to \infty}F(t_iy_0) \in K$.

The element $k(t)$ may depend on $t$ but, by our assumption
that $Tc \cap C$ is finite, we now see that any two points in 
$\pi^{-1}(c)$ may differ in at most finitely many coordinates.
The same of course holds for $\pi^{-1}(tc)$ for every $t \in T$.
Thus the map $\pi$ is asymptotic, so that the conditions 
in  Theorem \ref{asym} are fulfilled.
%0312 
%Thus we 
We have proved the following:

%mm  is this result related for establishibg the tameness in some cases in the work of Fuhrmann-Kwietniak \cite{F-K}  ? 
%
\begin{prop}
Let $F$ be a semi-cocycle as above. Then the
map $\pi$ is an asymptotic extension, whence (by Theorem \ref{asym})
the system $(X,T)$ is tame$_\mathbf{1}$.
\end{prop}

%-------------------------------------------

\br 
\section{A tame$_\mathbf{1}$ almost automorphic system which is not tame$_\mathbf{2}$}
\sk

\begin{prop} \label{p:AlmostAut}    \  
\begin{enumerate}
\item
There exists a tame$_\mathbf{1}$ almost automorphic $\Z$-system which is not tame$_\mathbf{2}$.
\item
There are two disjoint minimal tame$_{\bf 2}$ systems whose product
(which is minimal and tame$_{\bf 1}$) is not tame$_{\bf 2}$.
\end{enumerate}
\end{prop}

\begin{proof}
Let $\T=\R/\Z$ be the one-dimensional torus, and let $\al\in \R$
be a fixed irrational number and $T_\al \colon \T \to \T$ is the
rotation by $\al$,\ $T_\al \theta = \theta  +\al \pmod{1}$. We
define a topological space 
%060521 I add an intuitive explanation (referee's suggestion) 
$X=\T_A$ and a continuous map $\pi \colon X \to
\T$ as follows.
%zzz
%First of all, $\T_A$
Intuitively, $\T_A$ is a circularly ordered space which we get by splitting on the circle $\T$ the points of the orbit $A:=\{n \alpha: n \in \Z\}$. 
Then the map $\pi \colon X \to \T$ is just the natural projection (gluing back the splitted points). 
%Then 
%% 
For $\theta\in \T \setminus \{n\al : n \in \Z\}$ the
preimage $\pi^{-1}(\theta)$ 
%zzz
%will be 
is
a singleton $x_\theta$. On the
other hand for each $n\in \Z$, $\pi^{-1}(n\al)$ 
%will 
consists of
exactly two points $x^{-}_{n\al}$ and $x^{+}_{n\al}$. For
convenience we will use the notation $\theta^{\pm}$,\ ($\theta \in
\T$) for points of $X$, where $(n\al)^{-}=x^{-}_{n\al}$,\
$(n\al)^{+}=x^{+}_{n\al}$ and $\theta^{-}=\theta^{+}=x_\theta$ for
$\theta\in \T \setminus \{n\al : n \in \Z\}$. A basis for the
topology at a point of the form $x_\theta,\ \theta\in \T \setminus
\{n\al: n \in \Z\}$, is the collection of sets
$\pi^{-1}(\theta-\ep,\theta +\ep),\ \ep > 0$. For $(n\al)^{-}$ a
basis will be the collection of sets of the form $\{(n\al)^{-}\}
\cup \pi^{-1}(n\al-\ep,n\al)$, where $\ep>0$. Finally for
$(n\al)^{+}$ a basis will be the collection of sets of the form
$\{(n\al)^{+}\} \cup \pi^{-1}(n\al,n\al+\ep)$. It is not hard to
check that this defines a compact metrizable zero dimensional
topology on $X$ (in fact $X$ is homeomorphic to the Cantor set)
with respect to which $\pi$ is continuous. Next define $T \colon X\to X$
by the formula $T\theta^{\pm} = (\theta + \al)^{\pm}$.
%3010 
%Again it is
%not hard to see that $\pi: (T,X) \to (R_\al,\T)$ is a homomorphism
%of dynamical systems and that $(T,X)$ is minimal and not
%equicontinuous; in fact it is
%%nov3 not defined in the paper (so maybe some reference is good here)
%almost-automorphic.
%; see e.g. Veech \cite{V65}.
%In particular $(T,X)$ is not HNS.
It is easy to see that the system $(X,T)$ is minimal and that the natural factor map
$\pi \colon (T,X) \to (R_\al,\T)$ is asymptotic; so that the AA system $(X,T)$ is tame.

We now define for each $\ga \in \T$ two distinct maps
$p_\ga^{\pm} \colon X \to X$ by the formulas
%gl++
$$
p_\ga^+(\theta^{\pm})= (\theta + \ga)^+,
\qquad
p_\ga^-(\theta^{\pm})= (\theta + \ga)^-.
$$
We leave the easy verification of the following claims
%zzz
%as an exercise.
to the reader.
\begin{enumerate}
\item
For every $\ga\in \T$ and every {\em sequence}, $n_i
\nearrow \infty$ with
%0312   double $$
 $$\lim_{i\to\infty}n_i\al = \ga,\
{\text{and}}\ \forall i,\ n_i\al < \ga,$$  
we have
$\lim_{i\to\infty}T^{n_i} = p_\ga^{-}$ in $E(X,T)$.
An analogous statement holds for $p_\ga^{+}$.
\item
%2808  E(X,T) or E(T,X)  ? maybe simply E(X) ? 
$E(X,T) = \{T^n: n\in \Z\} \cup
\{p_\ga^{\pm}: \ga \in \T\}$
\item
The subspace $\{T^n: n\in \Z\}$ inherits from
$E$ the discrete topology.
\item
The adherence semigroup 
%e3110 adding double-dollar 
$$\Adh(X,T) = E(X,T) \setminus \{T^n: n\in \Z\} = \{p_\ga^{\pm}:
\ga \in \T\}$$ is homeomorphic to the ``two arrows" space of
Alexandroff and Urysohn (see \cite[page 212]{Eng}, and also to Ellis'
``two circles" example
%me ``two circles" has also another meaning
%``two circles" example
%me 5.29 and not 5.59
\cite[Example 5.29]{E-book}).
It thus follows that $E$ is a separable
Rosenthal compactum of cardinality $2^{\aleph_0}$.
\item
For each $\ga\in \T$ the complement of the set
$C(p_\ga^{\pm})$ of continuity points of $p_\ga^{\pm}$ is
the countable set $\{\theta^{\pm}: \theta + \ga = n\al,
\ {\text{for some}}\ n\in \Z\}$.
In particular, each element of $E$ is of Baire class 1.
\end{enumerate}

Algebraically 
$$
\Adh(X,T) = \{p_\ga^{\pm}: \ga \in \T\} = \T \times \{1, -1\}
$$ 
and it can be checked that its topology is
separable and first countable. Each of the two components $\T \times \{1\}$ and 
$\T \times \{-1\}$ is dense, and with the induced topology, it is the Sorgenfrey circle.
These components though are highly non-measurable sets.
Dynamically $M= \Adh(X,T)$ is the unique minimal left ideal of $E(X,T)$
and thus a minimal system (see Section \ref{mi} below).

Next let $\beta \in \R$ be an irrational number such that the pair $(\al, \beta)$ 
is independent over the rational numbers. We consider the dynamical system
$(Y,T)$ which is obtained as above from $ (\T, R_\beta)$.
The two systems  $(X,T)$ and $(Y,T)$ are then {\em disjoint}, which means that the product system
$(X \times Y, T \times T)$ is minimal.
It then follows that also the two minimal systems $\Adh(X,T)$ and $\Adh(Y,T)$ are disjoint,
so that the system
$$
\Om = \Adh(X,T)\times \Adh(X,T) \cong (\T \times \{1,-1\}) \times (\T \times \{1,-1\})
$$
is a minimal system.

Given $r \in E(X \times Y, T \times T)$ we let $(p,q) =(p_r,q_r) \in E(X,T) \times E(Y,T)$
be its canonical image in $E(X,T) \times E(Y,T)$. This then defines a homomorphism
of $\Adh(X \times Y, T \times T)$ into $\Om = \Adh(X,T) \times \Adh(X,T)$
which is clearly injective.
Conversely, given a pair $(p,q) \in \Adh(X,T)\times \Adh(X,T)$
there is, by minimality, an $r \in \Adh(X \times Y,T \times T)$ with $r(\Id_X, \Id_Y) = (p,q)$.
Thus our homomorphism is also surjective, i.e. an isomorphism
of $\Adh(X \times Y,T \times T)$ and $\Om = \Adh(X,T)\times \Adh(X,T)$.
%Clearly $\Om$ embeds as a closed subset of $E(X \times Y, T \times T)$.

%It is now easy to check that the subset 
%$$
%\Om_0 = (\T \times \{1\}) \times (\T \times \{-1\}) \subset \Om
%$$
%is a copy of the Sorgenfrey torus and that its subset
%$$
%\Del = \{(p_\ga^+, p_\ga^-)  : \ga \in \T\}
%$$
%inherits the discrete topology from $\Om$.
%\end{exa}

%Thus our system $(X \times Y, T \times T)$ is
%%2910 \begin{enumerate}
%\begin{itemize}
%\item
%a metric minimal $\Z$-system,
%\item
%almost automorphic (over the $2$-torus $(\T \times \T, R_\al \times R_\beta)$),
%\item
%tame ($E(X \times Y, T \times T)$ is a Rosenthal compactum),
%\item
%$E(X \times Y, T \times T)$ is first countable (e.g. by Theorem \ref{asym}),
%
%\item $E(X \times Y, T \times T)$ is not hereditarily separable. 
%%1211
%Indeed, as we have shown $\Adh(X \times Y)$ is the topological square of the two arrows space. Now observe that it contains the duplicate of the Cantor set. In order to see this consider the Cantor set $C \subset [0,1]$ and define in $\Adh(X \times Y)$ the following subset 
%$$
%Y:=C^{+} \times (1-C)^{+} \cup  C^{-} \times (1-C)^{+}.  
%$$
%It is homeomorphic to $D(\{0,1\}^\N)$ (here, $1-C:=\{1-x: x \in [0,1]\}$ and $x^{-} < x^{+}$). 
%\end{itemize}
%
%\sk 
%This completes our proof.

It is now easy to check that the subset 
$$
\Om_0 = (\T \times \{1\}) \times (\T \times \{-1\}) \subset \Om
$$
is a copy of the Sorgenfrey torus and that its subset
%1211d  
%$$
%\Del = \{(p_\ga^+, p_\ga^-)  : \ga \in \T\}
%$$
$$
\Del = \{(p_\ga^+, p_{-\ga}^+)  : \ga \in \T\}
$$ 
inherits the discrete topology from $\Om$.

\sk 
Thus our system $(X \times Y, T \times T)$ is
%2910 \begin{enumerate}
\begin{itemize}
\item
a metric minimal $\Z$-system,
\item
almost automorphic (over the $2$-torus $(\T \times \T, R_\al \times R_\beta)$),
\item
tame ($E(X \times Y, T \times T)$ is a Rosenthal compactum),
\item
$E(X \times Y, T \times T)$ is first countable (e.g. by Theorem \ref{asym}),

\item
and $E(X \times Y, T \times T)$ is not hereditarily separable. 
\end{itemize}

This completes our proof.
\end{proof}

\sk 

%1211c    After many hesiatations and mistakes this is the version, I think that it is true. 
%you was right. completely negative part is symmetric (in their union) to the completely positive. 
%One may isolate the negative pair in $Y$ by rectangles which are BELOW the anti-diagonal. 
%We can include this remark or remove
\begin{remark}
As we have shown $\Adh(X \times Y)$ is the topological square of the two arrows space. 
%1911
Note that it contains the Alexandroff duplicate $D(\{0,1\}^\N)$ of the Cantor set. 
It is equivalent to 
%1911
showing that the square $S^2$ of the split interval $S:=S([0,1])$ 
contains $D(\{0,1\}^\N)$ (here $S:=\{p^{\pm}: p \in [0,1]\}$ and $p^{-} < p^{+}$). 
In order to see this consider the Cantor set $C \subset [0,1]$ and define in $S^2$ the following subset 
%$$
%Y:=C^{+} \times (1-C)^{+} \cup  C^{-} \times (1-C)^{+}.  
%$$
$$
Y:=\{(p^{+},(1-p)^{+}): p \in C\} \cup  \{(p^{-},(1-p)^{+}): p \in C\}.  
$$
Then $Y$ is homeomorphic to $D(\{0,1\}^\N)$. 
%1911
The ``positive half"
 $\{(p^{+},(1-p)^{+}): p \in C\}$
  is a discrete subset of $Y$.   
\end{remark}

\br 
\section{Weak rigidity}\label{wr}
\sk

A dynamical system $(X,T)$ (where we assume that the $T$-action is effective) is 
said to be {\em weakly rigid} when $\Adh(X,T) = E(X,T)$ or equivalently when 
there is a {\bf net} 
%2510  e  or Id  ? 
$\{t_i\} \subset T \setminus \{e\}$ with $t_i \to \Id$ (pointwise on $X$).
It is called {\em rigid} when there is a {\bf sequence} 
$\{t_i\} \subset T \setminus \{e\}$ with $t_i \to \Id$ (pointwise on $X$).
Finally $(X,T)$ is {\em uniformly rigid} when the convergence of the sequence $\{t_i\}$
to $\Id$ is uniform. 
Note that a system $(X,T)$ is not weakly rigid when and only when $T$
is an open dense subset of $E(X,T)$.

The following claim follows directly from the fact that
the enveloping semigroup of a 
%2808 
metric 
tame system is a Fr\'{e}chet-Urysohn topological space (Theorem \ref{f:DynBFT}).

\begin{prop}
A tame system is weakly rigid if and only if it is rigid.
\end{prop}

\begin{prop}
For any dynamical system $(X,T)$, weak rigidity implies that $$\Asym(X,T) =\Delta.$$
\end{prop}

\begin{proof}
Suppose $(X,T)$ is weakly rigid and let $(x, x')$ be an asymptotic pair.
Let $\{t_i\} \subset T \setminus \{e\}$ be a net such that $t_i \to \Id$.
Then
$$
x = \lim t_i x, \quad x'  = \lim t_i x' 
$$
and $ \lim t_i x   = \lim t_i x'$, whence $x = x'$. 
\end{proof}

%
%
%\begin{claim}
%If a dynamical system $(X,T)$ is tame but not rigid then
%$\Asym(X,T) \not = \Delta$.
%\end{claim}

%0312 
\begin{exa}(See \cite{G-Ma}, \cite{GW} and \cite{AAB})
	\begin{enumerate}
		\item
		Every distal system is weakly rigid.
		\item
		Every WAP cascade is uniformly rigid.
		\item
		There are minimal weakly mixing uniformly rigid 
		%0412
		%cascaseds.
		cascades.
	\end{enumerate}
\end{exa}

\begin{exa}
%(1) Every WAP system is uniformly rigid.
The Glasner-Weiss examples (see \cite[theorem 11.1]{G-M}) 
are uniformly rigid cascades which are HNS hence tame.
These systems are metric, not equicontinuous and not minimal.
Being HNS their enveloping semigroups are actually metrizable (see \cite{G-M-U}).
\end{exa}

%0312  please read from here until the end of Section 5

\begin{exa}
	In \cite{H-J} the authors show that their `Auslander systems' (a family of
	almost automorphic systems which generalize the Floyd system) is never weakly rigid.
\end{exa}

\begin{question} \label{q-rigid} \
	\begin{enumerate}
		\item
		Is there an example of a tame rigid (metric, cascade) system $(X,T)$ which is not HNS (i.e. with
		non-metrizable enveloping semigroup)?
		\item
		Is there  a minimal, tame, rigid system $(X,T)$ which is not equicontinuous ?
	\end{enumerate}
\end{question}

\begin{exa} \label{e:EllisProj}  
	In \cite{E} the author shows that the action of 
	%2211
	$G =\mathrm{GL}(d,\R)$ on the projective space
	$\mathbb{P}^{d-1}, \ d \ge 2$, is tame and that the corresponding
	enveloping semigroup $E(\mathbb{P}^{d-1}, G)$ is not first countable. 
	He also shows that the group $G$ embeds as an open subset of  $E(\mathbb{P}^{d-1}, G)$.
	Thus this action is not weakly rigid. 
\end{exa}

%2808 
The system in the following example is a 
%2211
(two-to-one) extension of the system from Example \ref{e:EllisProj}. 
This demonstrates that {\bf the class Tame$_\mathbf{1}$ is not closed under factors}.

%2808 
\begin{ex} \label{e:Akin}  
	\cite{Ak-98} 
	The sphere $\mathbb{S}^{d-1}$ 
	with the projective action of the group $G=\mathrm{GL}(d,\R)$ (or, of the projective group 
	$\mathrm{PGL}(d,\R)$) is tame$_\mathbf{1}$ but not tame$_\mathbf{2}$.
\end{ex}
%\begin{proof}  
%	(1) Because $E(\mathbb{P}^{n-1})$ is Fr\'{e}chet but not first countable.
%	
%	(2) % In the latter case $E(\mathbb{S}^{n-1})$ is of type (ii)
%	Indeed, $E(\mathbb{S}^{n-1})$ is first countable but can not be hereditarily separable.
%	To see this observe that
%	if it were hereditarily separable then so would be its factor
%	$E(\mathbb{P}^{n-1})$, but this, in
%	turn, would imply that the separable Rosenthal compactum $E(\mathbb{P}^{n-1})$ is first countable, a contradiction (by (0)). 
%	
%
%\end{proof}

\begin{exa}\label{sub}
	Let $G=\mathrm{GL}(d,\R)$ and $X_1 = \mathbb{S}^{d-1}$, $X_2 = \mathbb{P}^{d-1}$.
	Set $X = X_1 \cup X_2$, 
	the disjoint sum of $G$-systems.  
	The dynamical system $(X,G)$ is tame and, as the system $(X_2,G)$ 
	is a factor of $(X_1,G)$, we have $E(X,G) = E(X_1,G)$.
	We see that the tame$_{\bf 1}$ system 
	%0312   REPLACING "T" by "G"      $(X,T)$ 
	$(X,G)$
	admits a subsystem, namely $(X_2,G)$, which 
	is not tame$_{\bf 1}$. Thus 
	{\bf the class Tame$_\mathbf{1}$ is not closed under subsystems}.
\end{exa}

%\noindent\hrulefill
%\vskip 0.5cm  
\br 
%0312 
\section{On tame dynamical systems which are not tame$_\mathbf{1}$} 
% whose enveloping semigroup is not first countable}
\sk

Let $(X,T)$ be a dynamical system with enveloping semigroup $E = E(X,T)$.
Let us call an element $p \in E$  a {\em parabolic idempotent with target $x_0$} if there is a point $x_0 \in X$
such that $px = x_0, \forall x \in X$, and a {\em loxodromic idempotent with target $(x_0, x_1)$} if there are
distinct points $x_0, x_1 \in X$ with $px = x_0, \forall x \in X \setminus \{x_1\}$ and $px_1 =x_1$.
We say that $x_0$ and $x_1$ are the {\em attracting and repulsing points of $p$} respectively.
Clearly, if $(X,T)$ admits a parabolic idempotent, it is necessarily a proximal system
and therefore contains a unique minimal set $Z \subseteq X$.
If $(X,T)$ is a proximal system then every minimal idempotent is parabolic with target at the minimal set 
(and of course conversely,
every parabolic idempotent whose target is in the minimal subset of $X$ is a minimal idempotent).

%27
\begin{prop}\label{two}
Let $(X,T)$ be a proximal dynamical system. 
%whose minimal subset $Z \subseteq X$ is infinite.
Let $Z \subset X$ be its (necessarily unique) minimal subset.
\begin{enumerate}
\item
Suppose that there is an uncountable set $B \subset X$ such that
for each $b \in B$ there is a loxodromic idempotent $p_b$ with target $(a_b, b)$, with 
$b$ as a repulsing point and $a_b \in Z$ the attracting point
such that $b \not = a_b$. 
Then $E(X,T)$ contains the uncountable discrete subset  
%2510 ADDING
$\{p_b : b \in B\}$,
hence it is not hereditarily separable.
\item
Suppose there is a point $a \in Z$ and an uncountable set 
 of points $B =\{b_\nu\} \subset   X \setminus \{a\}$
such that each pair $(a, b_\nu)$ is the target pair of a loxodromic idempotent $p_{(a, b_\nu)}$
with attracting point $a$ and a repulsing point $b_\nu$.
Then the parabolic idempotent $p_a$ defined by $p_ax = a, \ \forall x \in X$,
does not admit a countable basis for its topology, 
hence $E(X,T)$ is not first countable.
\end{enumerate}
\end{prop}

\begin{proof}
(1) \ Straightforward. 
%2510 It is easy to see that the set $\{p_b : b \in B\}$ is an uncountable discrete subset of $E(X,T)$.

(2) \ 
We repeat Ellis' argument in \cite{E} as follows:
Assuming otherwise, in view of Lemma \ref{countable}, there is a countable set $C \subset X$
such that for any $q \in E(X,T)$, if $qc = p_ac$ for every $c \in C$ then $q=p_a$.
Now the set $B$ is uncountable and we can choose an element $b_\nu \in B \setminus C$.
It then follows that for every $c \in C$ we have
$$
p_{(a,b_\nu)}c = p_a c =a, 
$$
but nonetheless $p_{(a,b_\nu)}b_\nu = b_\nu \not= a = p_ab_\nu$.
Thus $p_{(a,b_\nu)} \not = p_a$, a contradiction. 
\end{proof}

%0611
%0711
\begin{remark} \label{r:AlexComp} 
	%0811 REMOVING: Note that the (AVOIDING multiple appearance of "Note that") 
	The subset $\{p_a\} \cup B$ of $X$ is a topological copy of the
	Alexandroff compactification $A({\bf c})$, of a discrete space of size continuum.
	% (compare Theorem \ref{tri}). 
	Note that, in view of Theorem \ref{tri}, the existence of a copy of $A({\bf c})$ in $E(X)$ is another
	way of seeing that $(X,T)$ is not 
	%1411
	tame$_{\bf 1}$ 
	%This gives another explanation that $E(X,T)$ is not first countable.  
	%and $X$ is not tame$_\mathbf{1}$. 
	%This observation is actual anytime when
	%0811a 
	(see 
	%0811 
	additional cases where 
	 Proposition \ref{two}(2) applies, in Theorem \ref{uniform}, Theorem \ref{t:LinAreTame}
	 and Proposition \ref{ex:T}).  
\end{remark}

\br

Let $\Ga$ be an infinite countable uniform convergence group acting on a compact metric space
$X$ with infinite limit set $\La \subset X$. Then the system $(\La, \Ga)$ 
is minimal and every point of $\La$ is conical (see \cite[Proposition 3.3]{Bo}).

%2808 "wandering net", "collapsing sequence", $\Theta^0(X)$, are not defined

\begin{defin}
A point $b \in \La$ is a {\em conical limit point} if there is a wandering net, $\ga_n \in \Ga$,
such that for all $y \in \La \setminus \{b\}$, the ordered pairs $(\ga_n b, \ga_n y)$ lie in a compact subset of the
space of distinct pairs of $\La$.
\end{defin}

\begin{claim}
The following conditions are equivalent for a point $b \in \La$:
\begin{enumerate}
\item
It is a conical point.
\item
There is a collapsing sequence $\ga_n$ and a point $a \in \La$ with 
$\ga_n y \to a$ for all $y \in \La \setminus \{b\}$ and $\ga_n b \to b$.
\item
%A point $x \in X$ is a {\em conical limit point} if 
There is a point $y \in X, y \not = b$ and 
there are nets $x_n$  in $X \setminus \{b, y\}$ with $x_n \to b$,
and $\ga_n \in \Ga$,
such that the sequence
%there exists $y \in X \setminus \{x\}$ such that 
$\ga_n (b, y, x_n)$  remains in a compact subset of $\Theta^0(X)$.
\item
The point $b$ is the repulsing point for a loxodromic projection $p \in E(X,\Ga)$.
\end{enumerate}
\end{claim}

\begin{thm}\label{uniform}
Let $\Ga$ be an infinite countable uniform convergence group acting on a compact metric space
$X$ with infinite limit set $\La \subset X$.
Then the dynamical system $(X, \Ga)$ is tame but $E(X, \Ga)$  
\begin{enumerate}
\item[(i)]
is not hereditarily separable; 
in fact it contains a  subset of cardinality $2^{\aleph_0}$ which is discrete in the relative topology.
\item[(ii)] 
is not first countable,
\end{enumerate}
\end{thm}

\begin{proof}
The fact that $(X, \Ga)$ is tame follows directly from the definition of being a divergence group.
In  \cite{K-L} the authors show, with a direct simple argument, a stronger staetment,
namely that it is null.

(i) \  As every point of $X$ is conical, for each $b \in X$ let $p_b$ be a loxodromic idempotent,
say with repulsing point $b$. It is easy to check that every limit point of the 
collection $B = \{p_b : b \in X\}$ is parabolic. Thus the topology induced on $B$
from $E(X. \Ga)$ is discrete 
%2510 
(by Proposition \ref{two}.1). 

(ii) \ 
By \cite[Corollary II, 2.2]{G} every minimal proximal system is weakly mixing;
i.e. the product system $X \times X$ is topologically transitive, and $X$ being 
a metric space, the set of transitive points
$$
A = \{(x, y) \in X \times X : \overline{\Ga(x,y)} = X \times X\}
$$ 
is a dense $G_\delta$ subset of $X\times X$.
We then conclude, by the theorem of Kuratowski and Ulam,
that there is a residual subset $D$ of $X$ such that for every $x \in D$
the set $A_x = \{y \in X : (x,y) \in A\}$ is residual in $X$.
%$B = \pi(A)$, the projection of $A$ onto the first coordinate, is a
%comeager subset of $X$ and, moreover, for every 
In particular, $A_x$ is of cardinality $2^{\aleph_0}$.

Fix a point $(x_0, y_0) \in A$ and let $(z,w)$ be an arbitrary point in $X \times X$ with $z \not=w$.
There is the a sequence $\ga_n \in \Ga$ with $\ga_n(x_0, y_0) \to (z,w)$.
We can assume that $\ga_n $ is a collapsing sequence and then we have a pair
$(a,b) \in X \times X$ such that
$\ga_n x \to a$ for every $x \in X \setminus \{b\}$ and $\ga_n b \to b$.
Let $p_{(a,b)}$ denote the corresponding idempotent in $E(X, \Ga)$.
Thus $p_{(a,b)} x = a$ for every $x \in X \setminus \{b\}$ and $p_{(a,b)} b = b$.

It then follows that either
$z = \lim \ga_n x_0 = a$ and $w = \lim \ga_n =b$, 
or  $z = \lim \ga_n x_0 = b$ and $w = \lim \ga_n =a$.
Of course, by replacing $\ga_n$ by $\ga_n^{-1}$, we can assume with no loss of generality
that the first case occurs.

As a consequence of this discussion we conclude that:
\begin{quote}
For every $a \in D$ there is an a residual set of points $B_a =\{b_\nu\} \subset   X \setminus \{a\}$
such that each pair $(a, b_\nu)$ is the target pair of a loxodromic idempotent $p_{(a, b_\nu)}$
with attracting point $a$ and a repulsing point $b_\nu$.
\end{quote}
%
%We can now repeat Ellis' argument in \cite{E} as follows:
%Fix $a \in D$ as above and let $p_a$ be the parabolic minimal idempotent in $E(X, \Ga)$
%defined by $p_ax = a, \ \forall x \in X$.
%%Suppose that $\{U_n\}_{n \in \N}$ is a basis for the topology of $E(X,\Ga)$ at the point $p_a$.
%
%We claim that the point $p_a \in E(X,T)$ does not admit a countable basis for its topology.
%Assuming otherwise, in view of Lemma \ref{countable}, there is a countable set $C \subset X$
%such that for any $q \in E(X,T)$, if $qc = p_ac$ for every $c \in C$ then $q=p_a$.
%Now the set $B_a$ is uncountable and we can choose an element $b_\nu \in B_a \setminus C$.
%It then follows that for every $c \in C$ we have
%$$
%p_{(a,b_\nu)}c = p_a c =a, 
%$$
%but nonetheless $p_{(a,b_\nu)}b_\nu = b_\nu \not= a = p_ab_\nu$.
%Thus $p_{(a,b_\nu)} \not = p_a$, a contradiction.
Now apply Proposition \ref{two} (2).
\end{proof}

A Gromov hyperbolic group acting on its boundary is an example of a group satisfying
the conditions of Theorem \ref{uniform}.
In fact, it was shown by Bowditch \cite[Theorem 0.1]{Bo-98} that essentially these are the only examples:
If an  infinite countable group $\Ga$ is acting as a
uniform convergence group on a compact perfect metric space $X$ 
then, $\Ga$ is hyperbolic and the dynamical system $(X,\Ga)$ 
is isomorphic to the action of $\Ga$ on its Gromov boundary.

\begin{cor}\label{grom}
The action of a hyperbolic group $\Ga$ on its Gromov boundary 
$\partial \Ga$ is tame but not tame$_\mathbf{1}$.
%$E(\partial \Ga, \Ga)$ is not first countable.
\end{cor}

\begin{ex}\label{dynkin}(Dynkin and Maljutov \cite{D-M})
The free group $F_2$ on two generators, say $a$ and $b$, is hyperbolic and its 
boundary can be identified with the compact metric space $\Om$ (a Cantor set) of all
the one-sided infinite reduced words $w$ on the symbols $a, b, a^{-1}, b^{-1}$.
The group action is
$$
F_2 \times \Om \to \Om,   \quad (\ga, w) = \ga\cdot w,
$$
where $\ga\cdot w$ is obtained by concatenation of $\ga$ (written in its reduced form)
and $w$ and then performing the needed cancelations (see \cite{D-M}).
The resulting dynamical system is tame (in fact, null), and the 
%0312  Fr\'{e}chet-Urysohn Enveloping  semigroup $E(\Om, F_2)$ is not first countable.
enveloping semigroup $E(\Om, F_2)$ is Fr\'{e}chet-Urysohn but not first countable. 
\end{ex}

%3010
\begin{question}\label{q:abel}
	Suppose $T$ is abelian (or even more specifically that $T= \Z$). 
	Is there a tame metric (minimal) system $(X,T)$
	such that $E(X,T)$ is not first countable ?
	%2510  ADDING
	%3010
	%What if $X \subset \{0,1\}^{\Z}$ is a binary subshift ?
\end{question}

%\noindent\hrulefill
%\vskip 0.5cm 

\br
\section{ Linear actions on Euclidean spaces} \label{s:LinAct} 
 \sk 

 For many important actions $T \times X \to X$ the phase space is not compact. 
 One may still study the tameness of this action via $T$-compactifications of $X$. 
 It is especially natural whenever $X$ is locally compact. 
 Since $T$-factors preserve the class Tame (but not Tame$_\mathbf{1}$) 
 the question can be reduced to the case of the one-point $T$-compactification of $X$. 
 %\footnote{Meaning that it is true for every equivariant compactification of the present action. Enough to show this for the one-point compactification}
 
 \begin{thm} \label{t:LinAreTame} 
 	Every linear action $\GL(n,\R) \times \R^n \to \R^n$ %on the Euclidean space $\R^n$
 	 is tame but, for $n \geq 2$ it is not tame$_\mathbf{1}$.   % for every subgroup $G \subset GL(n,\R)$. 
 \end{thm}
 \begin{proof}
 	Let 
	%zzz
	$X_n:=\R^n \cup \{\infty\}$ be the one point compactification of $\R^n$ and $G \times X_n \to X_n$, 
	%zzz
	with $G = \GL(n,\R)$, be the canonical extension. 
	We have to show that $X_n$ is a tame $G$-system. 
 	For $p \in E(X,G)$ let
 	$$
	V_p : = p^{-1}(\R^n)=\{v \in \R^n : \ p(v) \in \R^n\}=\{v \in X : \ p(v) \neq \infty\}.
	$$
 	\sk 
 \noindent	{\bf Claim:}  The subset $V_p$ is a linear subspace (hence, closed) in $\R^n$ and the restriction 
	$p \rest V_p \colon V_p \to \R^n$ is a (continuous) linear map. 
 	\sk 
 	Indeed, let $g_i$ be a net in $G$ such that $\lim g_i=p$ in $E$. If $u,v \in V_p$ 
	then $\lim g_iu=p(u) \in \R^n$ and $\lim g_iv=p(v) \in \R^n$. 
	Then by the linearity of maps $g_i$ we obtain 
	$\lim g_i (c_1u+c_2v)=c_1p(u)+c_2p(v) \in \R^n$. Since  $V_p$ is finite dimensional, 
	the linear map $p \rest V_p  \colon V_p \to \R^n$ is necessarily continuous.  	\qed
 	
 	\sk
 	%mm I think one explanation is enough. I loke more the first with the cardinalities ...
 	%Using this claim we have two possibilities to complete the proof. 
 	Using this claim we obtain that the cardinality of $E(X_n)$ is not greater than $2^{\aleph_0}$. 
 	Now the tameness follows by the dynamical BFT dichotomy (Theorem \ref{f:DynBFT}).

To see that $(X_n,G)$ is not tame$_{\bf 1}$ we first observe that the following lemma holds.
	%mm We did not prove yet that tame$_1$ is hereditary.  Sounds true ! 
	The proof is by a verification of the conditions in Proposition \ref{two} (2).
	
	\begin{lem}
	%To see this consider the action of the subgroup 
	Let $S < GL(2,\R)$ be the subgroup 
	$$
S = \left\{
	\begin{pmatrix}
	a & b \\ 0 & 1
	\end{pmatrix}
	: a > 0, \ b \in \R 
	\right\}.
	$$
	Then the action of $S$ on the space 
	$$
	Y_1 =\left\{\left( \begin{array}{cc}
 			t  \\
 			1  
 		\end{array}  \right) :  \ t \in \R \right\} \cup \{\infty\}
		 \subset \R^2 \cup \{\infty\}
		$$  
	is tame but not tame$_{\bf 1}$.
	\end{lem}

Denoting, for each $s \in \R$
$$
	Y_s =\left\{\left( \begin{array}{cc}
 			t  \\
 			s  
 		\end{array}  \right) :  \ t \in \R \right\} \cup \{\infty\}
		 \subset \R^2 \cup \{\infty\},
		$$
we have $X_2 = \bigsqcup _{s \in \R}Y_s$ and it follows that
%0611  why ? and what is "Y" here ?
%0711  
$E(X_2,S) = E(Y_1,S)$.

Next observe that for any $n \geq 2$ the dynamical system $(X_2,S)$ is a subsystem 
of the system $(X_n, S)$
%zzz
%where $X_n = \R^n \cup \{\infty\}$, 
and we now consider 
$S$ as a subgroup of $GL(n,\R)$, embedded at the left top corner:
$$
A =
\begin{pmatrix}
	a & b \\ 0 & 1
\end{pmatrix}
\mapsto
\begin{pmatrix}
	A & {\bf{0}}_{2,n-2} \\ 
	{\bf{0}}_{n-2, 2} & {\bf {I}}_{n-2,n-2}
\end{pmatrix}
$$
Moreover, as under this embedding the action of $S$ on $X_n$ has the property that 
it acts as the identity on the last $n-2$ coordinates of a vector in $\R^n$,
it follows that $E(X_n, S) = E(X_2,S) = E(Y,S)$.
Now, since
% since $E(Y,S)$ is a subsemigroup of
%mm  
%$E(X_n,S) \subset E(X, SL(n, \R))$ we conclude that $(X, SL(n,\R))$  is not tame$_1$.
$E(X_n,S) \subset E(X_n, \GL(n, \R))$ we conclude that 
$E(X_n, \GL(n,\R))$  is also not tame$_{\bf 1}$.
 \end{proof}

 %mm 
 %2810
 The following lemma  presents another way of seeing that the action of 
 $\GL(2,\R)$ on $X_2 = \R^2 \cup \{\infty\}$ is not tame$_{\bf 1}$. 
 
 \begin{lem} \label{Xn}
 	The action $\GL(2,\R) \times X_2 \to X_2$ is not tame$_\mathbf{1}$. 
	%(where $X_2 = \R^2 \cup \{\infty\}$).
 \end{lem}
\begin{proof}
	 By Proposition \ref{countable} it is enough to show that there exists $p \in E(X_2)$ such that countable subsets $C(p)$ of $X_2$ cannot determine $p$. 
	 Take the following $\infty$-idempotent 
	 $$
	 p=p_{\infty} \colon X_2 \to X_2, \ p(x)=\infty \ \forall x \neq 0_2 \ \text{and} \ p(0_2)=0_2.
	 $$ 
	It really belongs to $E(X_2)$ because if $t_n \in \GL(2, \R)$ is the scalar matrix with 
	$t_n(x)=nx$ for every $x \in \R^2$, then $\lim t_n =p$ in $E(X_2)$. 	
	 Let, as before, $V_p : = p^{-1}(\R^2)$. Then for such $p$ we have $V_p=\{0_2\}$. 
	 Assume in contrary that $C(p)$ is a countable subset of $X_2$ which determines $p$. 
	 We can suppose that $C(p)$ does not contain $0_2$ (and also $\infty$). 
	 There exists a one-dimensional subspace $L$ in $\R^2$ which does not meet $C(p)$. 
	 %Moreover, there exists $q_L \in E(X_2)$ such that $q_L^{-1}(\R^2)=L.$ 
	 Moreover, there exists $q_L \in E(X_2)$ such that $q_L^{-1}(\R^2)=L$ and $q_L(x)=x$ for every $x \in L$.	 %mm 
	 Up to a conjugation, we can suppose that 
	 	$$
	 L =\left\{\left( \begin{array}{cc} 
	 	a  \\
	 	0  
	 \end{array}  \right) :  \ a\in \R \right\},   \ \ \ g_n= 	\begin{pmatrix}
	 1 & 0 \\ 0 & n
 \end{pmatrix},    \ \ \ q_L =\lim g_n. 
	 $$   
	 Then $p(x)=q_L(x)=\infty$ for every $x \in C(p)$ (but clearly, $p \neq q_L$).  
	\end{proof}

\sk 

%2810
\begin{remark} 
	The enveloping semigroup $E(X_n)$ of the action from Theorem \ref{t:LinAreTame}   
	can be identified with the semigroup of all partial linear endomorphisms 
	of $\R^n$.  
To see this observe that the claim from the proof of Theorem \ref{t:LinAreTame} can be reversed. 
Namely, every partial linear endomorphism $f \colon V \to W$ of $\R^n$ defines an element $p \in E(X_n)$ such that 
$$
V=p^{-1}(\R^n), \ p(V)=W, \ 
p(x)=f(x) \ \forall x \in V, \ p(y)=\infty \ \forall y \notin V.
$$
 Moreover, that assignment is a semigroup isomorphism: (partial) composition corresponds
to the product of suitable elements from the enveloping semigroup. 
\end{remark}

%0611
%0711
\begin{remark}
	Theorem \ref{tri} predicts that $E(X_2)$ in Lemma \ref{Xn} should contain the
	Alexandroff compactification $A({\bf c})$, of a discrete space of size continuum. 
	The following subset of $E(X_2)$ is indeed a topological copy of $A({\bf c})$
	%0811  REPLACING p_L by q_L (consistent with Lemma \ref{Xn})
	$$
	\{p_{\infty}\}\cup \{q_L: L \ \text{is a one-dimensional linear subspace in} \  \R^2\}.
	$$
\end{remark}

\br
%28a
%0611 Eli, you mentioned in email that you have a good candidate for the Cantor duplicate in this example. It would be interesting to include. What about Proposition \ref{p:AlmostAut} ?
 \begin{lem}\label{Y*}
	The action of $S$ on the space 
	$$
	Y_1^* =\left\{\left( \begin{array}{cc}
 			t  \\
 			1  
 		\end{array}  \right) :  \ t\in \R \right\} \cup \{\pm \infty\}
	$$  
	is  tame$_{\bf 1}$.
\end{lem}

%0711
\begin{proof}
Clearly every element 
$p \in \Adh(Y_1^*, S)$ 
fixes the points $\pm \infty$.
Suppose now that $p \in \Adh(Y_1^*, S)$ satisfies 
$p 
\begin{pmatrix}
s \\
1 
\end{pmatrix} =
\begin{pmatrix}
s \\
1 
\end{pmatrix} 
$
for some $s \in \R$.
It is then easy to check that it has either the form 
$p^o_{0s}$ or $p^i_{0s}$, where 
for $s \in \R$,
$$
p^o_{0s} 
\begin{pmatrix}
t \\
1 
\end{pmatrix}
= 
\begin{cases}
\infty & t >s \\
\begin{pmatrix}
s \\
1 
\end{pmatrix} 
& t = s  \\
-\infty & t < s,
\end{cases}  
\qquad \qquad \quad 
%0911
p^i_{s} 
\begin{pmatrix}
t \\
1 
\end{pmatrix}
= 
\begin{pmatrix}
s \\
1 
\end{pmatrix}
{\text{for all $t \in \R$}}.
$$
Applying $g \in S$ with 
$g
\begin{pmatrix}
s\\
1 
\end{pmatrix} =
\begin{pmatrix}
r \\
1 
\end{pmatrix}
$
we get the elements $gp^o_{0s} =p^o_{rs}$, where
$$
p^o_{rs} 
\begin{pmatrix}
t \\
1 
\end{pmatrix}
= 
\begin{cases}
\infty & t >s \\
\begin{pmatrix}
r \\
1 
\end{pmatrix} 
& t = s  \\
-\infty & t < s,
\end{cases}  
%\qquad \qquad \quad 
%p^i_{rs} 
%\begin{pmatrix}
%t \\
%1 
%\end{pmatrix}
%= 
%\begin{pmatrix}
%r \\
%1 
%\end{pmatrix}
%{\text{for all $t \in \R$}},
$$
and 
%$gp^i_{0s}= p^i_{rs} = p^i_{0r}$.
 %0811a
 %0911
 $gp^i_s = p^i_r$.
%$$
%p^i_{rs} 
%\begin{pmatrix}
%t \\
%1 
%\end{pmatrix}
%= 
%\begin{cases}
%\begin{pmatrix}
%s \\
%1 
%\end{pmatrix} 
%& t >s \\
%\begin{pmatrix}
%r \\
%1 
%\end{pmatrix} 
%& t = s  \\
%\begin{pmatrix}
%s \\
%1 
%\end{pmatrix} 
% & t < s.
%\end{cases}  
%$$

To complete the list of all the elements of $\Adh(Y_1^*,S)$ we have to add the limits:
%$p_\infty = \lim_{s \to -\infty} p^o_{0s}$, 
$$
p_\infty = \lim_{s \to -\infty} p^o_{0s}, \qquad 
p_\infty
\begin{pmatrix}
t \\
1 
\end{pmatrix}
= \infty
\quad {\text{for all $t \in \R$}},
$$
and
%$p_{-\infty}= \lim_{s \to \infty} p^o_{0s}$ 
$$
p_{-\infty}= \lim_{s \to \infty} p^o_{0s}, \qquad
p_{-\infty} 
\begin{pmatrix}
t \\
1 
\end{pmatrix}
= -\infty
\quad {\text{for all $t \in \R$.}}
$$
and finally
%1211d  \lim_{r \to \infty}
$$
p_{\pm\infty ,s}= \lim_{r \to \pm\infty} p^o_{rs}, 
\qquad
p_{\pm\infty} 
\begin{pmatrix}
t \\
1 
\end{pmatrix}
= 
\begin{cases}
\infty & t >s \\
\pm \infty
& t = s  \\
-\infty & t < s.
\end{cases}  
$$
%\begin{figure}[h]
%		\begin{center} \label{F**}
%			\scalebox{0.7}{\includegraphics{F*.pdf}}
%			\caption{A map of $E(Y_1^*,S)$}
%		\end{center}
%\end{figure}

%See Figure \ref{Fig2}. 
\begin{figure}[h] \label{Fig2} 
    \begin{center}
		\includegraphics[width=5cm,frame]{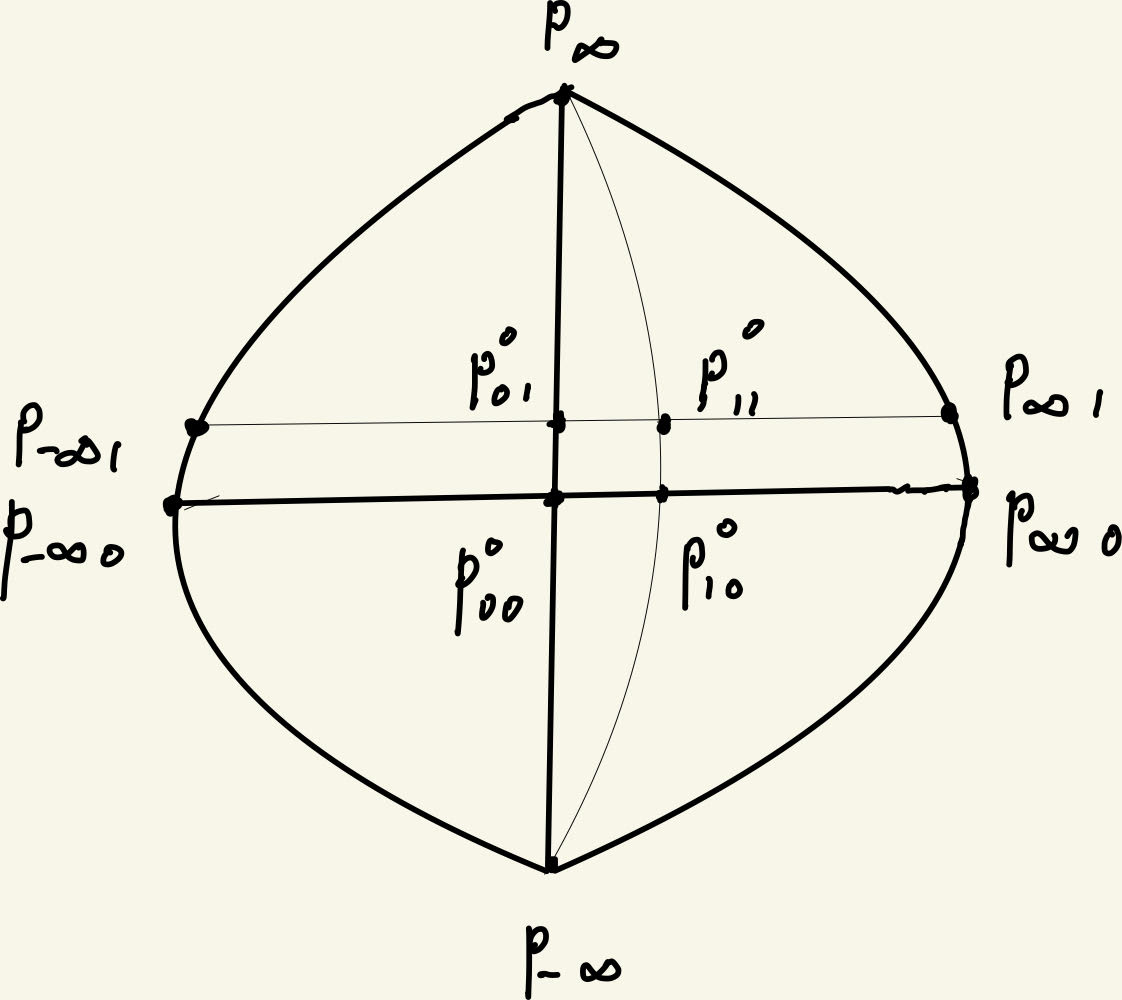}
	\end{center}
	\caption{A map of $E(Y_1^*,S)$}
\end{figure}

Now clearly each 
%1211d
$p^o_{rs}, \  p^i_{s}, \ r, s \in \R$, 
is determined by its value at three points,
while the elements $p_{\pm \infty}$ and $p_{\pm \infty, \infty}$
are determined by their values on
the set 
$$
\left\{
\begin{pmatrix}
n \\
1 
\end{pmatrix}
 : n \in \Z 
 \right\}.
 $$
Now apply Proposition \ref{countable}.
\end{proof}

As in Examples \ref{e:EllisProj} and \ref{e:Akin}
we observe here a situation 
where a tame$_{\bf 1}$ system $(Y_1^*,S)$ admits the system $(Y_1,S)$ as a factor,
and the latter has an enveloping semigroup which is not first countable.

\br  
%1208
\section{Order preserving actions} \label{s:Order} 
\sk 

%060521  This "introductory part" to Section 8 was essentially changed 
%implementing referee's suggestions. So please read it until next %060521
%zzz
%The most 
A standard example of a circularly ordered space is the circle $\T$. 
%It is well known that 
An abstract circular (some authors prefer the name, \textit{cyclical}) order $R$ on a set $X$ 
%formally 
can be defined as a certain ternary relation.
Intuitively, it is a linear (i.e., total) order which has been bent into a ``circle".  
 For formal classical definitions and related new dynamical topics see  \cite{GM-c,GM-int}.

Let $(X,T)$ be a compact dynamical system. 
If $X$ is a circularly ordered (c-ordered) space, 
compact in its interval topology, 
%space 
and every $t$-translation $X \to X$ is circular ordered preserving (COP) 
then we say that $(X,T)$ is a \textit{circularly ordered system}. 

The class of  \textit{linearly ordered dynamical systems} (LOP) is defined similarly. 
%Linear order preserving maps we abbreviate as LOP. 
Note that every linearly ordered \textbf{compact} $T$-system is a circularly ordered $T$-system. 
%The most 
A prototypic example of a circularly ordered system is the circle $\T$ 
equipped
with the action of 
the group
$H_+(\T)$ (or, some of its subgroup) on $\T$. 
Several important Sturmian like symbolic dynamical systems
%, in fact, 
are circularly ordered dynamical systems, hence tame. 
This is so because every circularly ordered system is tame. See the details 
in \cite{GM-c}. 

Let $X, Y$ be linearly (circularly) ordered sets. Below we denote by $M_+(X,Y)$ the set of 
%all 
linear (circular) order preserving maps from $X$ into $Y$. Similarly, by $H_+(X)$ we denote the  group of all LOP (resp., COP) homeomorphisms $X \to X$. 
%If  
When $X$ is compact, 
%then 
$H_+(X)$ is a topological group carrying the usual compact open topology.

Our results show that
there are naturally defined circularly ordered tame systems of all three types. 
Moreover, the (pointwise compact) space $M_+(X,Y)$ is first countable for every compact metrizable linearly ordered spaces $X, Y$ (Theorem \ref{t:lin-Helly}). 
This result is a generalized version of the Helly theorem (for $X=Y=[0,1]$).  

This 
fact
implies that every linearly ordered metric dynamical $T$-system is tame$_\mathbf{1}$ (Theorem \ref{t:LOT1}). 
In contrast, the circular version of the Helly space $M_+(\T,\T)$ is not first countable (Remark \ref{r:c-Helly}). 
This 
%topological 
result demonstrates once more 
%one more times, 
the relative complexity of circular orders 
%against 
when compared with linear orders. 
%060521   until here 

\br 
%060521 \subsection{
\textbf{Orderable enveloping semigroups and the Tame$_\mathbf{2}$ class}
%\textbf{Orderable enveloping semigroups and Tame$_\mathbf{2}$ class}
\sk 

%2510
By \cite{LB} in linearly ordered topological spaces (LOTS) the separability is hereditary. That is, in a separable LOTS every topological subspace is separable. This
%zzz 
can be easily extended to the case of c-ordered spaces, using the 
%2005 \textit{covering space} (which is linearly ordered). 
\textit{covering space}, \cite{GM-c} (which is linearly ordered).

\begin{lem}
	Let $R$ be a  circular order on $X$ such that $X$ in its interval topology is separable.  Then every topological subspace of $X$ is separable. 
\end{lem} 
\begin{proof} 
	%	More precisely, let $(X,R)$ be a circularly ordered space. Then 
	By \cite[Prop. 2.10]{GM-c} for every $c \in X$ there exists a LOTS $X(c):=([c^-,c^+],\leq_c)$ such that $X(c)$ we get from $X$ splitting the point $z$ into two new points $z^-, z^+$. 
%3010	
%We have is 
The natural projection $q\colon [c^-,c^+] \to X$ (identifying the endpoints of $X(c)$) is a c-order preserving quotient map. 
	Here, the linear order $\leq_c$ on $X(c)$ is defined by 
	$$z^{-} \leq_z a \leq_z z^{+}, \ \ \ \  a <_z b \Leftrightarrow [z,a,b] \ \ \ \forall a \neq b \neq z \neq a, $$
	where $[z,a,b]$ indicates the fact that $(z,a,b) \in R$ (the triple is positively oriented). 
	% (where, $z^-$ is the least element and $z^+$ is the greatest element). 
	%This linear order restores the original circular order. Meaning that $R_{\leq_z}=R$. 

	Moreover, it is easy to see that the restriction of $q$ on every interval $[a,b]_{\leq_z}$ onto $[a,b]_{\circ}$  %where $z \neq a$ (eq., $z \notin [a,b]_{\leq_z}$) 
	is %both a homeomorphism and 
	an isomorphism of linearly ordered sets, where $c^- \neq a$ or $c^+ \neq b$. 
	%When $X$ is compact then so is $X(c)$. 
	The topologies on $X$ and $X(c)$ are both interval topologies. It follows that if $D$ is a topologically dense subset of $X$ then the same set $D$ is dense in $[c^-,c^+]$. Hence, if the c-ordered space $X$ is separable then
%3010
the LOTS $X(c)$ is also separable. Since $X(c)$ is hereditarily separable and $q \colon [c^-,c^+] \to X$ is continuous onto, we obtain that $X$ is also hereditarily separable. 
\end{proof}

In particular, if the enveloping semigroup is c-ordered then its separability is hereditary. In view of our Definition \ref{d:Tame_k} this leads to the following sufficient condition.

\begin{thm}  \label{t:EisCODS} 
	If the enveloping semigroup $E(X)$, as a compact space, is circularly ordered %dynamical system 
	then the original metrizable dynamical system $X$ is tame$_\mathbf{2}$. 
\end{thm}

%For example, the enveloping semigroup of the classical Sturmian cascades is a circularly ordered not metrizable compact, \cite{GM-c}. So we get 

\begin{cor} \label{cor:St}  %Sturmian cascades $X$ are tame$_2$  (but not HNS). 	More generally, 
	The Sturmian like cascades  
	$\Split(\T, R_\al; A)$ are tame$_\mathbf{2}$ (but not HNS).  	
\end{cor}  
%\textit{First proof:} 
\begin{proof}
We can apply Theorem \ref{t:EisCODS} because for these systems the enveloping semigroup 
$E$ becomes (by \cite[Cor. 6.5]{GM-c}) a circularly ordered cascade, where $E=\T_{\T} \cup \Z$ is a c-ordered compact (nonmetrizable) subset of 
the c-ordered lexicographic order $\T \times \{-,0,+\}$. % (see \cite{GM-c}).  
\end{proof}

%\textit{Second proof:} 
%by \cite[Cor. 6.5]{GM-c} for every irrational $\alpha \in \R$ and every $R_\alpha$-invariant subset 
%$A \subset \T$, the enveloping semigroup of $Split(\T, R_\al; A)$ is the
%circularly ordered cascade $\T_{\T} \cup \Z$, where $\T_{\T}$ is homeomorphic to the two arrows space and $\Z$ is a discrete copy of the integers.
%This space $\T_{\T} \cup \Z$ is hereditarily separable. It is not metrizable, whence $X$ is not HNS. 
%
%\textit{Third proof:} Apply Theorem \ref{asym}. 

\sk 
\begin{question} \label{q:c-ord}
	For which c-ordered systems the enveloping semigroup is c-ordered (at least as a compact space) ? 
%	\footnote{\textcolor{blue}{A sufficient condition can be found in our new project ``Orderable groups and semigroup compactifications":  \textbf{Theorem.} 
%			Let $K$ be a linearly ordered compact $G$-system 		
%			and $\leq$ be a linear order on $G$ such that every orbit map $\tilde{x}: G \to K, g \mapsto gx$ is linear order preserving for every $x \in K$.  
%			%$g_1 <_G g_2$ iff $g_1x < g_2x$ for every $x \in K$. 
%			Then the Ellis semigroup $E(K)$ is a linearly ordered semigroup}}% and the Ellis compactification $j: G \to E(K)$ is a linearly ordered compactification. }} 
\end{question}

\br
%\subsection{
	
\textbf{Generalized Helly theorem and Tame$_\mathbf{1}$ class}
%\textbf{Generalized Helly theorem and Tame$_\mathbf{1}$ class}
\sk 

It is a well known fact in classical analysis that every order preserving bounded real
valued  function $f \colon X \to \R$ on an interval $X \subseteq \R$ has one-sided limits. 
This can be extended to the case of linear order preserving functions 
$f \colon X \to Y$ such that $X$ is first countable and $Y$ is sequentially compact; see \cite{FS}. The following lemma is easy to verify.

\begin{lem} \label{l:propSIDE} Let $p \colon X \to Y$ be LOP where $X$ and $Y$ 
%3010
%be 
are
separable metrizable LOTS (linearly ordered topological spaces) and $Y$ is compact.
Then
	\begin{enumerate}
		\item $p$ has one-sided limits $p(a^-), p(a^+)$ at any $a \in X$ and $p(a) \in [p(a^-), p(a^+)]$. 
		%		\item If $p$ is constant on some interval $(x,a]$ for some $x \neq a$ then %$p(a^-)$ exists and 
		%		$p(a^-)=p(a)$. Similarly, if $p$ is constant on some interval $[a,x)$ for some $x \neq a$ then %$p(a^-)$ exists and 
		%		$p(a)=p(a^+)$. 
		
		%		 \item 	If $[a,x)$ is open for some $x \neq a$ 
		%		 %iff $a$ is left singular (i.e., there exists $a_n$ such that )
		%		 then $p(a^-)=p(a)$. 
		\item $p$ is continuous at $a$ if and only if $p(a^-)=p(a^+)$ if and only if $p(a^-)=p(a^+)=p(a)$. 
		
		\item 
		$p$ has at most countably many discontinuity points. 
	\end{enumerate} 
\end{lem} 
% \begin{proof} 
% 	%	Every separable linearly ordered compact is first countable. Every compact linearly ordered is sequentially compact.  
% 
% \end{proof}

\br 
Let $(X,\leq)$ be a linearly ordered set. Let us say that 
$u \in X$ is a \textit{singular point} 
%0605 
and write $u \in sing(X)$, 
if there is a maximal element below $u$ or a minimal element above it. 
It is equivalent to saying that 
%there exists $v  \neq u$ in $X$ such that either 
%$[u,v)$ is open or $(v,u]$ is open. 
%This, in turn, is equivalent to the following condition: 
%%e3110 saying 
%%0111 it should be "saying"
%%say  
for every $v \neq u$, either $[u,v)$ or $(v,u]$ is a clopen set.

%0605 replacing (twice) \circ by \leq   (referee also missed this like us :-))
\begin{lem} \label{l:sing}  Let $(X,\leq)$ be a linearly ordered compact metric space. 
	Then $sing(X,\leq)$ is at most countable. 
\end{lem}
\begin{proof}
	For every $u \in sing(X)$ choose exactly one clopen (nonempty) set $[u,v)$
	or $(v,u]$, where $v \neq u$. 
	This assignment defines a  1-1 map from $sing(X)$ 
	%3110 to  
	into the set $clop(X)$ of all clopen subsets of $X$. 
	Now observe that $clop(X)$ is countable for every compact metric space (take a countable 
	basis $\mathcal{B}$ of open subsets; 
	then every clopen subset is a finite union of some members of $\mathcal{B}$).
\end{proof}

Recall that the Helly space $M_+([0,1],[0,1])$ is first countable 
(see for example, \cite[page 127]{SS}).  
The following theorem %\ref{t:lin-Helly} 
was inspired by this classical fact.

\begin{thm}[Generalized Helly space] \label{t:lin-Helly}
	Let $X$ and $Y$ be linearly ordered sets with their interval topologies such that 
	$X$ and $Y$ are compact metric spaces. Then the (compact) set $M_+(X,Y)$ of all LOP maps is first countable in the pointwise
	%3110 convergence 
	topology. %\footnote{probably, not for $BV(X,[a,b])$} 
\end{thm}
\begin{proof}  Let $p \in M_+(X,Y)$. 
	Lemma \ref{l:propSIDE} guarantees that $p$ has at most countably many discontinuities. 
	Denote this set by $disc(p)$.  %By Lemma \ref{l:nonReg} the set $D_f$ is countable. 
	%Denote by $cont(p)$ the subset of continuity points of $p$ in $X$. 
	%That is, $cont(p)=X \setminus disc(p)$.  
	Since $X$ is compact and metrizable, there exists a countable dense subset $A$ in $X$. 
	By Lemma \ref{l:sing} the set $sing(X)$ is countable. 
	Consider the following countable set %(which depends on $p$)
	$$
	C:=disc(p) \cup sing(X) \cup A.
	$$
	It is enough to show that $C$ satisfies the condition of Lemma \ref{countable} with $E:= M_+(X,Y)$.	
	So, let $q \in M_+(X,Y)$ be such that $q(c)=p(c)$ \ $\forall x \in C$. We have to show that 
	$q(c)=p(c)$ \ $\forall x \in X$. 
	
	Assuming the contrary, let $q(x_0) \neq p(x_0)$ for some $x_0 \in X$. Then, clearly, 
	$x_0 \in cont(p)$. 
	Choose $y_1,y_2 \in Y$ such that $p(x_0) \in (y_1,y_2)$ and 
	$$q(x_0) \notin (y_1,y_2).$$
	There exists an open neighborhood $O$ of $x_0$ in $X$ such that $p(O) \subset (y_1,y_2)$. 
	Since open intervals form a topological basis of the topology of a circular order, 
	we can suppose that $O$ is an open interval $(x_1,x_2)$ containing $x_0$. 
	By our choice, $x_0$ is not singular. Therefore, $(x_1,x_0)$ and $(x_0,x_2)$ are nonempty. Since $A$ is dense in $X$, there exist $a_1,a_2 \in A$ such that $a_1 \in (x_1,x_0)$ and $a_2 \in (x_0,x_2)$. Then 
	$a_1< x_0 < a_2$. So, $q(a_1) \leq q(x_0) \leq q(a_2)$. 
	Since $q(a_1)=p(a_1), q(a_2)=p(a_2)$ and $p(a_1), p(a_2) \in (y_1,y_2)$ it follows that 
	$y_1 < q(x_0) < y_2$. 
	We get 
	$$
	q(x_0) \in (y_1,y_2). 
	$$
	This contradiction completes the proof. 
\end{proof}

\sk 
\begin{thm} \label{t:LOT1}
	Every linearly ordered compact metric dynamical system is  tame$_\mathbf{1}$.
\end{thm}
\begin{proof}Let $X$ be a compact metrizable linearly ordered dynamical system. 
	%The enveloping semigroup $E(X)$ is the pointwise closure of the set of all $s$-translations on $X$. 
	Every element $p \in E(X)$ is a LOP selfmap $X \to X$, because 
	$M_+(X,X)$ is pointwise closed. 
	%by Theorem \ref{t:M_Closed}. 
	So, $E$ is a subspace of $M_+(X,X)$ which is first countable by Theorem \ref{t:lin-Helly}. 
\end{proof}

\begin{ex} \label{ex:Helly} 
	Consider the linearly ordered system $([0,1], H_+([0,1]))$. 
	The enveloping semigroup of this c-order preserving system is a (compact) subspace of the Helly space which is first countable. So, this system is tame$_\mathbf{1}$. It is not tame$_{\bf 2}$. 
	In fact, it is (like the Helly space) not hereditarily separable. 
	There exists a discrete uncountable subspace in the enveloping semigroup. 
	For each $z \in [0,1]$ consider the functions 
	$$
	f_z \colon [0,1] \to [0,1], \ \ \ f_z =
	\begin{cases}
		0 & {\text{for}} \ x < z\\
		\frac{1}{2} & {\text{for}} \ x =z\\
		0 & {\text{for}} \ x > z.
	\end{cases}
	$$	
	Then $\{f_z \colon z \in (0,1)\}$ is an uncountable discrete subset of $E(X)$. 
\end{ex}

 Recall again that every separable linearly 
    
\br
%\subsection{
\textbf{Circularly ordered systems, in general, are not tame$_\mathbf{1}$}
%\textbf{Circularly ordered systems, in general, are not tame$_\mathbf{1}$}
\sk

\begin{prop} \label{ex:T} 
	Let $H_+(\T)$ be the Polish topological group of all c-order preserving homeomorphisms of the circle $\T$. The minimal circularly ordered dynamical system $(\T,H_+(\T))$ is tame but not tame$_\mathbf{1}$. 
\end{prop}
\begin{proof}
	 This system is tame being a circularly ordered system, \cite{GM-c}. Let us show that the enveloping semigroup of this c-order preserving system is not first countable. 
	Choose any point $a \in \T$. For every $b \neq a$ in $\T$ 
	%of points $B =\{b_\nu\} \subset   \T \setminus \{a\}$
	the pair $(a, b)$ is the target pair of a loxodromic idempotent $p=p_{(a, b)}$
	with attracting point $a$ and a repulsing point $b$. 
		%See Figure 2. 
	
	%2510  adding Figure2  (I drew up it using "Paint") 
	\begin{figure}[h]
		\begin{center} \label{F2}
			\scalebox{0.7}{\includegraphics{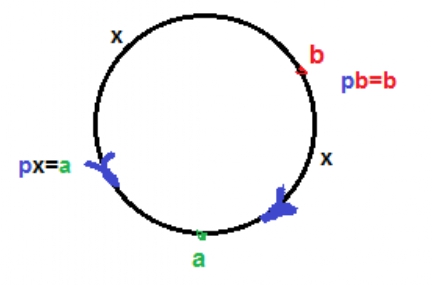}}
			\caption{loxodromic idempotent $p$}
		\end{center}
	\end{figure}

	Then, according to Proposition \ref{two}.2, the parabolic idempotent $p_a$ defined by $p_ax = a, \ \forall x \in \T$, does not admit a countable basis for its topology. 
	%hence $E(X,T)$ is not first countable.
	This example, in particular, shows that a \textit{null system} need not be tame$_\mathbf{1}$. 
\end{proof}

\sk 
\begin{remark} \label{r:c-Helly} 
	%2510 {\color{red}Theorems ?} 
	Theorems \ref{t:lin-Helly} and \ref{t:LOT1} cannot be extended, in general, to circular orders. Indeed, the ``circular analog of Helly's space" $M_+(\T,\T)$ (which is a separable Rosenthal compactum) is not first countable. Also its subspace, the enveloping semigroup $E(\T)$ of the circularly ordered system 
	%e3110 $(H_+(\T), \T)$   
	$(\T, H_+(\T))$
	from Example \ref{ex:T}, is not first countable. 
\end{remark}

%2910
\begin{remark} \label{r:int}
%3010 
	Let $\Q_0=\Q/\Z \subset \T$ be the circled rationals with the discrete topology and $G=\Aut(\Q_0)$  the automorphism group (in its pointwise topology) of the circularly ordered set $\Q_0$. In \cite{GM-int} we 
	have shown that 
	the universal minimal $G$-system of the Polish topological group $G$ can be identified with the circularly ordered compact metric space $\mathrm{Split}(\T; \Q_\circ)$, which is constructed by splitting the points of $\Q_0$ in $\T$. 
	%e3110 is this phrase "In a similar to " correct ? 
In a similar 
%0111
way to the proof of Proposition \ref{ex:T} one shows that
the $G$-system $M(G)$ is tame but not tame$_\mathbf{1}$. 
\end{remark}

\br

%3010
In view of Examples \ref{e:EllisProj}, \ref{e:Akin},  Lemmas \ref{Xn} and \ref{Y*}, and 
the examples in the present section, we pose the following.

\begin{question}\label{q:21}
Let $(X,T)$ be a metric, point transitive, tame system which is not tame$_{\bf 1}$.
Is there always an extension $\pi \colon (X^*,Y) \to (X,T)$ such that
$(X^*,T)$ is metric, point transitive and tame$_{\bf 1}$ ?
Is there such an extension $\pi$ which is also at most two-to-one ?
\end{question}

It is though worth mentioning that all of these examples are connected to ordered structures,
which may be the secrete behind the existence of the tame$_{\bf 1}$ extension.

%\noindent\hrulefill
%\vskip 0.5cm 
\br 
\section{Minimal left ideals}\label{mi}
\sk

According to Ellis' theory every enveloping semigroups $E(X,T)$ of a dynamical system $(X,T)$
contains at least one minimal left ideal. These minimal left ideals are the same as the minimal subsets
of the dynamical system $(E(X,T), T)$.
A minimal left ideal $M$ has the following form:
The set $J$ of idempotents in $M$ is nonempty.
We choose an arbitrary idempotent $u \in J$ and set $G = uM$.
Then $G$ is a group, $M = J \cdot G = \bigcup_{v \in J} vG$, and the representation of an element $p \in M$
as $p = v\al$, with $v \in J$ and $\al \in G$ is unique.
For each $\al \in G$, the map $R_\al \colon p \mapsto p\al$ is a homeomorphism of $M$ onto itself
which commutes with the (left) action of $T$ on $M$, so that $R_\al$ is an automorphism
of the dynamical system $(M,T)$.
Moreover, the map $\al \mapsto R_\al$ is a group isomorphism between $G$ and $\Aut(M,T)$.

A minimal system $(X,T)$ admits a single left minimal ideal if and only if the proximal relation $P \subset X \times X$ is
an equivalence relation, but in general even a minimal metrizable system can admit $2^{\cb}$ 
minimal left ideals (see \cite{G-G}).
It may happen that the proximal relation $P$ on a minimal system $(X,T)$ 
is an equivalence relation but is not closed (see \cite{S}).
The relation $P$ is closed if and only if 
$E(X,T)$ has a unique minimal ideal $M$ and $J$ is a closed subset of $M$.

Recall that a metric minimal tame dynamical system $(X,T)$ which admits an invariant measure
is almost automorphic \cite{G-18}, and in particular its proximal relation is a closed equivalence relation.
Thus the enveloping semigroup of such a system always admits a unique minimal left ideal $M$
and its subset of idempotents $J$ is a closed subset of $M$.
In fact, in many such systems it may happen that this unique minimal left ideal coincides with the 
adherence semigroup $A(X,T)$.

Let $(X,T)$ be a metric minimal tame system admitting a $T$-invariant probability measure. 
Then the system $(X,T)$ is almost automorphic and the extension $\pi \colon X \to~Y$,
where $Y$ is the largest equicontinuous factor of $X$, is an almost one-to-one map.
Let, as before 
$$
X_0 =\{x \in X : \pi^{-1}(\pi(x)) = \{x\}\}, \quad
Y_0 = \pi (X_0) = \{y \in Y : |\pi^{-1}(y)| = 1\}.
$$

\sk 
\begin{prop}\label{mli}
Let $(X,T)$ be a metric minimal tame system admitting a $T$-invariant probability measure. 
Then the enveloping semigroup $E(X,T)$ contains
a unique minimal left ideal $M \subset E(X,T)$, and the space 
$M$ is a separable Rosenthal compactum.
If $Y \setminus Y_0$ is countable,
%3110
i.e. when  $(X,T)$ is AA$_c$, then $M$ is first countable.
\end{prop}

\begin{proof}
It is shown in \cite{G} that the minimal system $(X,T)$ is almost automorphic; i.e it is an almost one-to-one 
extension of its largest equicontinuous factor $(Y,T)$.
It then follows that the proximal relation on $X$ is a closed equivalence relation and thus
$E(X,T)$ admits a unique minimal left ideal $M$ and its set of idempotents $J$ is closed.
Being a closed subset of a Rosenthal compactum $M$ is also a Rosenthal compactum.
Each $T$-orbit in $M$ is dense and thus $M$ is separable.

The dynamical system $(Y,T)$ is minimal and equicontinuous and therefore has the form
$Y = K/L$ where $K$ is a compact second countable topological group and $L \le K$
is a closed subgroup. The enveloping semigroup of the system
$(Y,T) = (K/L,T)$ is the system $(K,T)$ and it follows
that there is a homomorphism 
$$
\tilde\pi \colon (M,T) \to (K,T),
$$
which is a proximal extension.
it also follows that the restriction of $\tilde\pi$ to $G$ is a group isomorphism
$$
\tilde\pi \rest G \colon G \to K.
$$
We also have $\tilde\pi^{-1}(e) = J$, where $e$ is the identity element of $K$
and it follows that $J$ is a closed $G_\del$ subset of $M$.

We observe that it suffices to show that each $v \in J$ is a $G_\del$ point of $J$.
In fact, as $\{R_\al : \al \in G\} \cong \Aut(M,T)$ it will follow that every
$p = v\al \in vG$ is also a $G_\del$ point of $M$.

Now it follows from our assumption the $Y \setminus Y_0$ is countable that 
in fact $J$ is metrizable. For each $y \in Y \setminus Y_0$ we choose a point 
$x_y \in \pi^{-1}(y)$ and define the evaluation map 
$$
{\rm{eva}}_y \colon v \mapsto vx_y, \quad J \to \pi^{-1}(y).
$$
Now the countable collection of continuous maps $\{{\rm{eva}}_y\}_{y \in Y \setminus Y_0}$
separates points on $J$ and our claim follows.
This completes the proof of the proposition. 
% So we fix $v \in J$ and, by complete regularity of $J$,
% unless $J$ is a single point, there is a continuous function $f: J \to [0,1]$ such that 
% $f (v)=0,$ and $f(w)=1$ for some $w \in J$. 
% Then $J_1 = f^{-1}([0,\frac{1}{2}])$ is a compact $G_{\delta}$ set in $J$. 
% Now repeat this procedure at each ordinal $\nu < \om_1$ , to define 
% a closed $G_\del$ set $J_{\nu+1} \subsetneq J_\nu$, and at a limit ordinal
% $\theta$ set $J_\theta = \cap \{J_\nu : \nu < \theta\}$.
% Now by tameness we have $\card{J} \leq 2^{\aleph_0}$ and,
%as we assume that $2^{\aleph_0} =\om_1$, we conclude that
\end{proof}

\sk 
\begin{remark}\label{metricJ}
If in the above proposition we replace the condition that $Y \setminus Y_0$ be countable
by the assumption that the set $J$ of minimal idempotents is metrizable, we obtain the same result.
\end{remark}

\begin{exa}   \ 
\begin{enumerate}
\item
In the case of a hyperbolic group acting on its Gromov boundary the unique minimal ideal
is isomorphic to the action on the boundary, hence actually metrizable (see Corollary \ref{grom}).
\item
The Sturmian system is an example where $\Adh(X,T)$ is minimal
and hence it coincides with the unique minimal left ideal. Topologically it is the split circle
%$S(\Sbf^1)$ 
(see \cite[Example 14.10]{G-M}). In this example $|J|=2$.
\item
For any system which satisfies the assumptions of Theorem \ref{asym} 
we have that $\Adh(X,T)$ is minimal and hence coincides with the unique minimal left ideal.
\item
For the Floyd system $(X,T)$ we have
$$
E(X,T) \supsetneq \Adh(X,T) \supsetneq M,
$$
where $M$ is the unique minimal left ideal in $E(X,T)$ (see \cite{H-J}).
\item
In the example \cite[Example 8.43]{GM-tLN} we see that  
$\Adh(X,T) =E(X,T) \setminus 	\Z = M$ is the unique minimal left ideal of $E(X,T)$.
Algebraically it is the product set $\T^d \times \mathcal F$, where $\mathcal F$ is the collection of
ordered orthonormal bases for $\R^d$. It is easy to see that $J = \Fcal$ and that
the topology induced on $J$ is actually the natural compact metric topology on $\Fcal$.
Thus, by Remark \ref{metricJ} the system $(X,T)$ is tame$_\mathbf{1}$.
%3110
Note that this AA system is not AA$_c$.
\end{enumerate}
\end{exa}

%\vskip 0.5cm 

%\noindent\hrulefill

%\vskip 0.5cm 
\br  
\section{Are the groups $vG, \ v \in J$ always homeomorphic to each other ?}
\sk

For the Sturmian system $(X,T)$ (or any other system of the form $\Split(\T, R_\al; A)$), the 
Adherence semigroup $\Adh(X,T)$ is the Ellis' two circles system and it coincides with
the unique minimal left ideal $M$ of  $E(X,T)$.
The set $J$ of minimal idempotents in $M$ consists of two elements 
$$
J = \{u_+, u_{-}\}.
$$ 
If we choose $u_+ =u$ then the group $G = uM$ is algebraically isomorphic to the circle group $\T$,
and as a topological space it is homeomorphic to the Sorgenfrey circle.
The map $g \mapsto - u_{-}g$ is a topological isomorphism from $G$ onto $u_{-}G$.

\br

We next consider a similar but more complicated example where the same phenomenon holds
%27
(compare with Subsection \ref{semi}).

\begin{exa}\label{cos}
Let $\al \in (0,1)$ be an irrational number.
We start our construction with the dynamical system $(\T, R_\al)$ defined on the circle $\T = \R/\Z$
by $R_\al \theta = \theta + \al \ \pmod  1$.
Let

$$
F(\theta) = 
\begin{cases}
	\cos(\frac{2\pi}{\theta}) & 0 < \theta < 1/2 \\
	2\theta & 1/2 \leq \theta < 1\\
	0 & \theta =0
\end{cases}
$$
\begin{figure}[h]
	\begin{center}
		\includegraphics[width=7cm,frame]{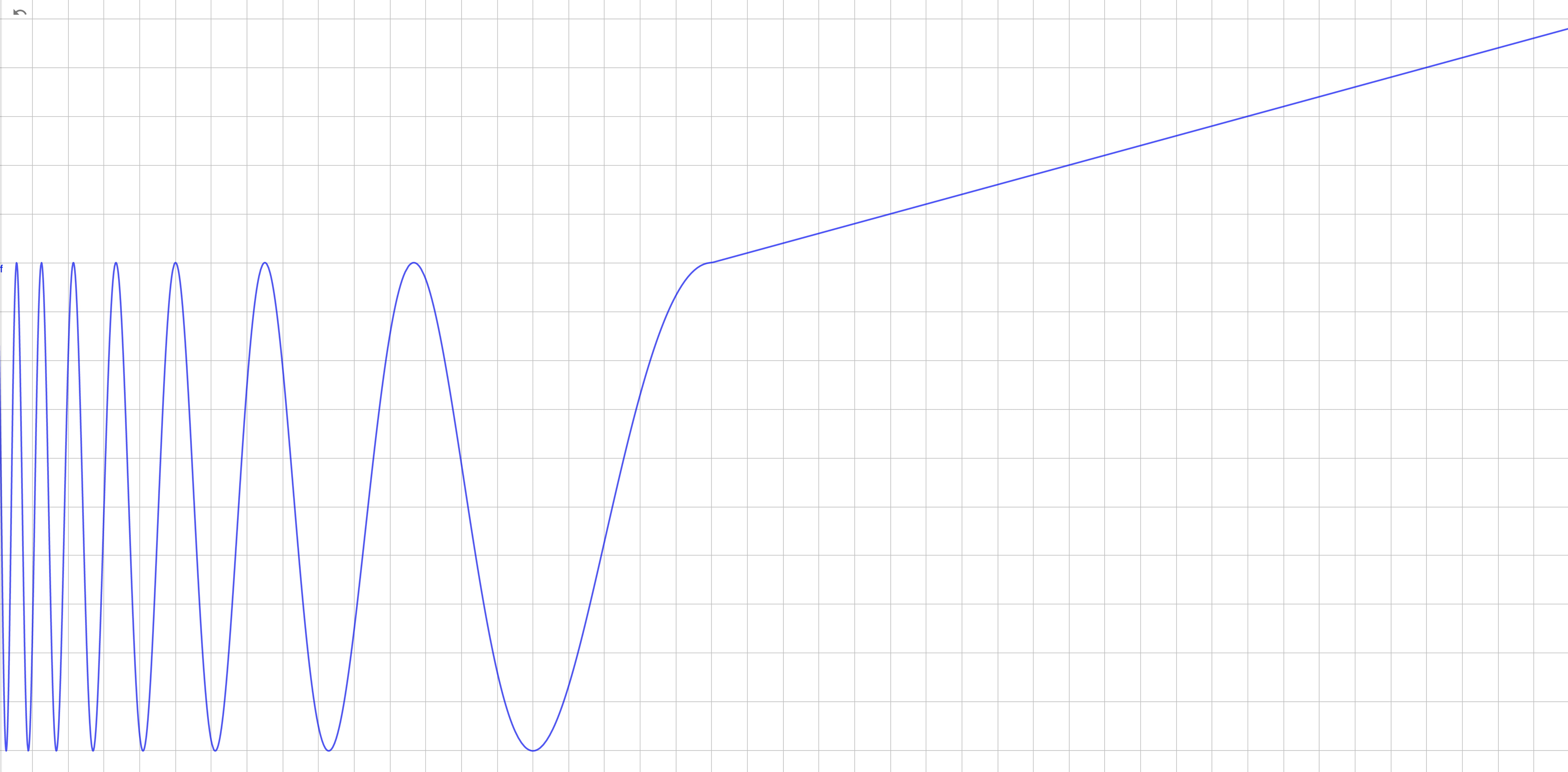}
	\end{center}
	\caption{The graph of the function $F(\theta)$}
\end{figure}
a function on $\T$ with $0$ as the unique discontinuity point.
Set $f(n) = F(n\al)$, a function in $\Om =[-1, 2]^\Z$.
%We consider $F$ as a function on $\T = \R/\Z$.
Let $x_0 \in \T \times [-1,2]^\Z$ be defined by
$$
x_0(n) = (0, F(n\al)) = (0, f(n)), \quad n \in \Z.
$$
Let $T \colon \T \times \Om  \to \T \times \Om$ be defined by
$$
T(\theta, \om) = (\theta +\al, S\om), 
$$
where for $\om \in \Om$ the shift homeomorphism $S \colon \Om \to \Om$ 
is defined by $S\om(n) = \om(n+1)$.
Let 
$$
X = \cls \Ocal_T(x_0) = \overline{\{T^n x_0 : n \in \Z\}} \subset  \T \times \Om.
$$
We let $\pi \colon X \to \T$ denote the projection on the $\T$ coordinate.
Clearly $\pi \colon (X,T) \to (\T, R_\al)$ is a homomorphism of dynamical systems.
One can also easily see that $\pi$ is an almost one-to-one extension, with
$\pi^{-1}(\pi(x)) = \{x\}$ for every $x \in X$ such that $\pi(x) \not \in \Ocal_{R_\al}(0)$.
In fact for a sequence $n_i$ such that $\lim n_i\al = \theta \not \in \Ocal_{R_\al}(0)$, we have:
$$
\lim_{i \to \infty} T^{n_i}x_0 = x, \quad \text{with} \quad 
x(n) = (\theta, F(\theta + n\al)),\quad (n \in \Z).
$$
It then follows that the system $(X,T)$ is minimal and almost automorphic.
Let $n_i \in \Z$ be a sequence such that $n_i \al \to 0$ with $n_i \al \in [1/2,1)$.
Assume that $x_2 := \lim T^{n_i}x_0$ exists. It is then easy to see that $x_2$ has the form
$$
x_2(n)  = 
\begin{cases}
	(0, f(n))  & n \not = 0\\
	(0, 2) & n=0.
\end{cases}
$$
On the other hand if $n_i \in \Z$ is a sequence such that
$n_i \al \to 0$ with $n_i \al \in (0, 1/2)$  and such that $x_t := \lim T^{n_i}x_0$ exists, then 
$x_t$ has the form 
$$
x_t(n)  = 
\begin{cases}
	(0, f(n))  & n \not = 0\\
	(0, t) & n=0,
\end{cases}
$$
for a unique $t \in [-1,1]$.

This gives us a full description of the fiber over $0$:
$$
\pi^{-1}(0) = \{x_2\} \cup \{x_t : t \in [-1,1]\}.
$$
Of course this defines also $\pi^{-1}(n\al)$ for every $n \in \Z$ and, as
$\pi^{-1}(\theta)$ is a singleton for every $\theta$ which is not in the $R_\al$-orbit of $0$,
we now have a complete description of all the points of $X$.

As a base for the topology at $x_2$ we can take the collection $\{V_k(x_2): \  k =2,3,\dots\}$, where
$$
V_k(x_2) =\pi^{-1}(1- 1/k, 1) \cup \{x_2\}.
$$
The collection $\{V_{k,l}(x_t): \  k, l  =2,3,\dots\}$, where
$$
V_{k,l}(x_t)   = \left(\pi^{-1}(0, 1/k) \cap \{x : x(0) = (\pi(x),s)\  \& \  |s -t| < 1/l\}\right) \cup \{x_s : |t -s| < 1/l\},
$$
will serve as a basis for the topology at $x_t \ (0 \le t \le 1)$.

Of course we have $T^nV_k(x_2) = V_k(T^nx_2)$ and 
$T^nV_k(x_t) = V_k(T^nx_t)$, for $n \in \Z$. Finally for a point $x \not \in \pi^{-1}(\Ocal_{R_\al}(0))$
the collection $\{V_k(x), \  k =2,3,\dots\}$, where
$$
V_k(x) =\pi^{-1}(\pi(x)- 1/k, \pi(x) +1/k),
$$ 
is a basis for the topology at $x$.

\br

Next note that, as the elements of $\pi^{-1}(0)$ differ from each other only at one coordinate,
the extension $\pi$ is asymptotic. It then follows that our system $(X,T)$ is tame and that
$E(X,T)$ is first countable (see Theorem \ref{asym}).
Thus we proved the following:

\begin{claim}
The system $(X,T)$ is AA$_{c}$ (but not AA$_{cc}$) and it is tame$_{\bf 1}$.
\end{claim}

\br

We will next determine the nature of the elements of the
adherence semigroup $\Adh(X,T) = E(X,T) \setminus \{T^n : n \in \Z\}$.

We start by determining the idempotents in $\Adh(X,T)$.
If $v = v^2$ is such an idempotent, then clearly $\pi_*(v) = \Id_\T = 0$,
where $\pi_* \colon E(X,T) \to E(\T, R_\al) \cong \T$ is the canonical homomorphism induced by $\pi$,
and where $E(\T, R_\al)$ is identified with $\T$. Now if $T^{n_i} \to v$, then it follows that
$n_i\al \to 0$, and thus $T^{n_i}x_0 \to vx_0 = x_\ep$, with either $\ep=2$ or $\ep = t \in [-1,1]$.
We now note that this determines $v$, since then $vx =x$ when $x \not \in \Ocal_T(x_0)$
and $vT^jx = T^jvx = T^j x_\ep$ for every $x \in \pi^{-1}(0)$ and $j \in \Z$. We then write $v = v_\ep$.
Conversely if $n_i$ is a sequence in $\Z$ such that $n_i\al \to 0$ and 
$T^{n_i}x_0 \to x_\ep$, then we must have $T^{n_i} \to v_\ep$ in $E(X,T)$.
Note that clearly $v_\ep$ is a minimal idempotent.

Let now $p = \lim T^{n_j}$ be an element of $\Adh(X,T)$, with $\lim n_j\al = \ga \in \T$.
%If $\ga = k \al$ for some $k \in \Z$ then 
For $x \in \pi^{-1}(-\ga)$ we have $px = x_\ep$ for some $\ep$ and it follows that
$p(\pi^{-1}(-\ga)) = \{x_\ep\}$. 
We now check and see that $v_\ep p =p$.
It follows immediately that $\Adh(X,T) = M$, the unique minimal left ideal of $E(X,T)$.
Denote by $J$ the set of minimal idempotents in $\Adh(X,T)$; thus 
$$
J = \{v_2\} \cup \{v_t : t \in [-1,1]\}.
$$ 
Note that $v_\ep v_\eta = v_\ep$ for every $\ep$ and $\eta$.
For convenience let $u = v_2$ and we set $G = uM$.
With this notation we have that for $p$ as above,  $p = v_\ep p$ and  denoting $g_\ga = up \in G$,
we get $p = v_\ep g_\ga$. Thus
$$
G = uM = \{g_\ga : \ga \in \T\}  \cong \T,
$$
and every $p \in M$ has uniquely the form $p = v_\ep g_\ga$ with $g_\ga \in G$ and $v_\ep \in J$.

Given $p \in \Adh(X,T)$ with $\pi_*(p) = \ga$, in order to determine the unique $\ep$ for which
$v_\ep p =p \in v_\ep G$, we have to evaluate $px(0) = (0, \ep)$ for some (any) $x \in \pi^{-1}(-\ga)$.

The collection $G$ is actually a group which algebraically is isomorphic to $\T$,
with $g_\ga \leftrightarrow \ga$, and the map $r_\ga \colon p \mapsto pg_\ga$
defines a (continuous) automorphism of the minimal system $(M, T) = (\Adh(X,T), T)$.
Thus $G \cong \Aut(\Adh(X,T),T) = \{r_\ga : \ga \in \T\}$.

\br

As the topology of $A(X,T)$ is the topology of pointwise convergence, it is now easy to see
that the topology induced on $J$ is such that $u = u_2$ is an isolated point of $J$ and that the
collection $\{v_t : t \in [-1,1]\}$ is naturally homeomorphic to the interval $[-1, 1]$.

\br

For $x \in X$ let $\eva_{x} \colon  \Adh(X,T) \to X$ denote the evaluation map defined by
$$
\eva_x(p) = px, \quad (p \in \Adh(X,T)).
$$
Clearly the collection of sets 
$$
\{\eva_x^{-1}(V) : x \in X, \ {\text{$V$ open in $X$}}\}, 
$$  
forms a sub-base for the topology of $\Adh(X,T)$.

We will next show that $G$, with the topology it inherits from $\Adh(X,T)$, is homeomorphic
to the Sorgenfrey circle (i.e. the half open intervals topology on $\T$). 

\br

We now have that when $x$ ranges over $X$, and we write $\pi(x) = -\theta$,
sets of the form:
\begin{gather*}
	G \cap \eva_{x}^{-1}(V_k(x_2)) = 
	\{g_\ga \colon \ga - \theta \in (1 -1/k, 1)\}   \cup \{g_\theta\} \cong (\theta - 1/k,\theta], \\
	G \cap \eva_{x}^{-1}(V_{k,l}(x_t)) = 
	\{g_\ga \colon \ga - \theta \in (0, 1/k)\}  \cong (\theta,\theta + 1/k), \\
	{\text{and for $x' \not\in \pi^{-1}(\Ocal_{R_\al})$ with $\pi(x') = \theta'$}}\\
	G \cap \eva_{x}^{-1}(V_k(x')) = 
	\{g_\ga \colon \ga - \theta \in (\theta' -1/k, \theta' + 1/k)  
	\cong (\theta' + \theta - 1/k,\theta' + \theta + 1/k), \\
	\quad k , l =2,3,\dots, \ t \in [-1,1],
\end{gather*}
together with all their $T^n$-translates, form a basis for the topology of $G$.
This completes the proof
of the claim that the relative topology on $G$ is the 
left-half-open interval Sorgenfrey topology on $\T$.

Similarly we see that for each $t \in [-1,1]$ the topology induced on the group $v_tG$
is the right-half-open interval Sorgenfrey topology on $\T$.

\br

Thus for $t, s \in [-1,1]$ the  map
$$
L_{v_t} \colon v_s G \to v_tG,   \quad  v_sg_\ga  \mapsto  v_tg_\ga
$$
is a topological isomorphism, while for $u$ and $t$ the map
$$
L_u \colon v_t G \to G,    \quad  v_tg_\ga  \mapsto - g_\ga
$$
is a topological isomorphism.
\end{exa}

\begin{question} \label{pr:2-idemp}
	Is there a minimal dynamical system $(X,T)$ such that in %0312 the 
	$E(X,T)$ there is
	a minimal left ideal $M$ and two idempotents $u, v \in M$ such that
	the groups $uM$ and $vM$ (with the induced topology from $M$) are not homeomorphic ?
\end{question}

%\vskip 0.5cm 
\br  
%\noindent\hrulefill
\section{The $\beta$-rank of a tame system}
\sk

%Define, for 
%$f : E \to \R$, the {\em oscillation function}
%$$
%\osc(f, x) = \inf \{\sup_{x_1, x_2 \in V} |f(x_1) - f(x_2)| : V \subset E \ {\text{open}} \ , x \in V\}
%$$ 
%and similarly, for $P$ a subset of $E$,
%$$
%\osc(f, x, P) = \osc(f \rest P, x).
%$$ 
%
%Consider, for each $\ep >0$, the derivative operation 
%$$
%P \mapsto P^*_{\ep,f} = \{x \in P : \osc(f, x, P) \ge \ep\}
%$$
%and by iterating define $P^\al_{\ep,f}$ for $\al < \om_1$, and let
%$$
%\eta(f, \ep, P) = 
%\begin{cases}
%{\text{least}} \ \al \ {\text{with}}\ P^\al_{\ep,f}= \emptyset, & {\text{if such an $\al$ exists}}\\
%\om_1 & {\text{otherwise}},
%\end{cases}
%$$
%Set
%$\beta(f, \ep) = \beta(f, \ep, E)$ and finally define the 
%{\em oscillation rank} $\beta(f)$ of $F$ by
%$$
%\beta(f) = \sup_{\ep > 0} \beta(f,\ep).
%$$
%
%\br

Let $(X,T)$ be a metric (tame) dynamical system, $E = E(X,T)$ its enveloping semigroup.
For $p \in E$ define the the {\em oscillation function of $p$ at $x \in X$} as
$$
\osc(p, x) = \inf \{\sup_{x_1, x_2 \in V} d(px_1, px_2) : V \subset X \ {\text{open}} \ , x \in V\}
$$ 
and similarly, for $A$ a subset of $X$ with $x \in A$,
$$
\osc(p, x, A) = \osc(p \rest A, x).
$$ 

Consider, for each $\ep >0$, the derivative operation 
$$
A \mapsto A'_{\ep,p} = \{x \in A : \osc(p, x, A) \ge \ep\}
$$
and by iterating define $A^\al_{\ep,p}$ for $\al < \om_1$.
Let
$$
\beta(p, \ep, A) = 
\begin{cases}
{\text{least}} \ \al \ {\text{with}}\  A^\al_{\ep,p} = \emptyset, & {\text{if such an $\al$ exists}}\\
\om_1 & {\text{otherwise}},
\end{cases}
$$
Set
$\beta(p, \ep) = \beta(p, \ep, X)$ and define the 
{\em oscillation rank} $\beta(p)$ of $p$ by
$$
\beta(p) = \sup_{\ep > 0} \beta(p,\ep).
$$
Note that $p$ is continuous iff $\beta(p) = 1$.

Finally define the {\em $\beta$-rank of the system $(X,T)$} as the ordinal
$$
\beta(X,T) = \sup \{\beta(p) : p \in E\}.
$$

The following claim follows from \cite{Bour-80}, and 
%1811
%\cite[Theorem 3, Theorem 8]{K-L} 
\cite[Theorem 3, Theorem 8]{Ke-L} 
(where the 
$\al, \beta$ and $\ga$ ranks are shown to be essentially equivalent for bounded Baire class 1 functions).

\begin{claim}
For a metric tame dynamical system $(X,T)$ we have
$\beta(X,T) < \om_1$.
\end{claim}

\begin{claim}
For the Sturmian system we have $\beta(X,T) = 2$.
\end{claim}

\begin{proof}
In order to compute $\beta(u_+)$ we observe that for a fixed $\ep >0$ there are only
finitely many points, all of the form $x=T^nx_0^-$, where $\osc(u_+, x) \ge \ep$.
Thus $\beta(u_+, \ep) =2$ for every $\ep >0$, hence $\beta(u_+)=2$.
Similarly $\beta(u_-)=2$ and, as $\beta(pg_\ga)=\beta(r_\ga(p)) = \beta(p)$
for every $p \in \Adh(X,T)$ and $\ga \in G$, we conclude that $\beta(p) =2$
for every $p \in \Adh(X,T)$. Of course $\beta(T^n) = 1$ and we conclude that $\beta(X,T) =2$,
as claimed.
\end{proof}

A similar proof will show that:

\begin{claim}
For  the ``$\cos(\frac{1}{x})$-system" $(X,T)$ (Example \ref{cos}) we have
$\beta(X,T) < \om_1$.
\end{claim}

\begin{claim}
The Dynkin-Maljutov system $(\Om, F_2)$ (Example \ref{dynkin}) has $\beta$-rank $2$.
\end{claim}

\begin{proof}
We have $\beta(t) =1$ for every $t \in F_2$, $\beta(p) =1$ for a parabolic $p \in \Adh(\Om, F_2)$,
and $\beta(p) =2$ for a loxodromic $p$. Thus $\beta(\Om, F_2) =2$ as claimed.
\end{proof}

\br

The Sturmian system is obtained by considering a cover of $\T$ by the two closed sets
$[0, \al]$ and $[\al, 1]$. In \cite{Pa} Paul shows that, in fact, 
every almost automorphic {\bf symbolic} minimal system arises from 
a finite separating cover of its maximal equicontinuous factor.
Both functions $\ch_{[0,\al]}$ (in the Sturmian system) and the function 
$F$ (in the $\cos(\frac{1}{x})$-system) are semicontinuous, hence with $\beta$-rank $2$
(as Baire class 1 functions),
and, as we have seen, the $\beta$-rank of these dynamical systems is $2$.
We will show next that for 
%1811
%a 
some separating cover $A_0, A_1$ of the circle by two closed sets
(so that again the $\beta$ rank of the function $\ch_{A_0}$ is $2$),
the corresponding almost automorphic system $(X,T)$ is not even tame,
so that $\beta(X,T) = \om_1$.

The following propositions and theorem are proved in 
\cite[Proposition 3.5, Proposition 3.6, Corollary 3.7]{FGJO}.

\begin{prop}
Let $(\T, R_\al)$ be the rotation system with $\al$ irrational and let $B \subset \T$ be a residual subset. 
Let $A_0, A_1 \subset \T$ be two closed subsets that satisfy
\begin{enumerate}
\item
$\cls\inte(A_0) = A_0$ and $\cls\inte(A_1) = A_1$.
\item
$\inte(A_0) \cap \inte(A_1) = \emptyset$.
\item
$C = \partial A_0 = \partial A_1$ is a Cantor set.
\end{enumerate}
Then there exists an infinite subset $I \subset \Z$ such that for every $a \in \{0,1\}^I$,
there is a point $x \in B$ with $T^ix  = x +i\al \in \inte(A_{a(i)})$ for every $i \in I$.
\end{prop}

Now let $F= \ch_{A_1}$ i.e. 
$$
F(\theta) =
\begin{cases}
1 & \theta \in A_1,\\
0 & {\text{otherwise}}.
\end{cases}
$$
Choose a point $y_0 \in \inte{A_1}$ and let $x_0 = f \in \{0,1\}^\Z$
be defined by $x_0(n) = f(n) = F(y_0 + n\al), \ (n \in \Z)$.
We set $X = \cls \{T^n x_0 : n \in \Z\}$, where $T$ is the shift map on $\{0,1\}^\Z$.

The dynamical system $(X,T)$ is almost automorphic with a maximal equicontinuous factor $(\T, R_\al)$,
where the factor map $\pi \colon (X,T) \to (\T, R_\al)$ is an almost one-to-one extension and, in particular,
$\pi^{-1}(y_0) = \{x_0\}$. It is also a cut and project system as follows (see \cite{Pa}, \cite{BJL}):

\begin{prop}
The almost automorphic subshift $(X,T)$ is isomorphic to the cut and project dynamical system 
$(\Om(\curlywedge(W)),\Z)$ obtained from the 
cut and project scheme $(\Z,\T,L)$, with lattice $L = \{(n, n\al) : n \in \Z\} \subset \Z \times \T$, where
\begin{itemize}
\item
$0 \in\T$ has a unique preimage under the factor map $\beta \colon (\Om(\curlywedge(W)),\Z) \to  (\T, R_\al)$.
\item 
$W = \beta([1])$, where $[1] = \{\xi \in \{0,1\}^\Z  : \xi(0) =1\}$.
\end{itemize}
Moreover, the window $W$ is proper, that is, $\inte(W ) = W$ .
\end{prop}

Now in this situation the following theorem applies.

\begin{thm}
Suppose that $(\Z,\T,L)$ is a cut and project scheme.
If $W$ is a proper window and 
$\partial W$ is a Cantor set, then the system $(\Om(\curlywedge(W)),\Z)$
admits an infinite free set $I \subset G$ i.e.
for every $a \in \{0,1\}^I$,
there is a point $x \in \Om(\curlywedge(W))$ with 
$x(i) =  a(i)$ for every $i \in I$.
Thus the system $(\Om(\curlywedge(W)),\Z)$ is not tame.
\end{thm}

\sk 

\begin{question} \label{pr:rank}
Is there, for every ordinal $\al < \om_1$, a tame metric system $(X,T)$ with $\beta(X,T) = \al$ ?
Or, to be more modest, find a dynamical system with $\beta(X,T) =3$.
\end{question}

 %2510  is it tame_1 (tame_2)   ? 
 \br 
\section{More open questions} 
\sk

%3010

%\begin{question}
%	Suppose $T$ is abelian (or even more specifically that $T= \Z$). 
%	Is there a tame metric (minimal) system $(X,T)$
%	such that $E(X,T)$ is not first countable ?
%	%2510  ADDING
%	%3010
%	%What if $X \subset \{0,1\}^{\Z}$ is a binary subshift ?
%\end{question}

%\sk 
%\begin{remark}
A separable Banach space $V$ is Rosenthal if and only if its \textit{enveloping semigroup} 
$\E(V)$ is a (separable) Rosenthal compactum  
(see the BFT dichotomy for Banach spaces \cite[Theorem 5.22]{GM-AffComp}). 
%\end{remark}

%3010
%\begin{remark} \ 
	 
Another characterization, due to 
%	By a characterization of 
Odell and Rosenthal is, that a separable Banach space $V$ is Rosenthal if and only if every element 
$x^{**} \in B^{**}$ is a Baire class 1 function on $B^*$, iff $B^{**}$ is a separable Rosenthal compactum. 
One can ask what is the nature of the spaces $B^{**}$  in terms of the
 Todor\u{c}evi\'{c} classes mentioned in Theorem \ref{tri}. 
	
	Now, as  noted by Bourgain \cite{Bour77} the point $0$ is not a $G_{\delta}$ subset of $B^{**}$ for any Rosenthal space which is not Asplund. In particular, it follows that $B^{**}$ is not first countable (at $0$). 
	%$B^{**}$ is never first countable in such cases. So, it is impossible to develop a hierarchy of Rosenthal  spaces based on properties of $B^{**}$. 
	This implies that the enveloping semigroup $\E(V)$ of $V$ is not hereditarily separable (because, by \cite[Lemma 2.6.2]{GM-AffComp} the weak-star compact ball $B^{**}$ is a continuous image of $\E(V)$). 
%\end{remark}

\begin{question} \label{q:Ban}
	Is there a separable Rosenthal Banach space $V$ such that $\E(V)$ is first countable but not metrizable (equivalently, $V$ is not Asplund)~?
\end{question}

\begin{question} \label{q:dendr} 
	Every continuous topological group action on a dendrite $D$ is tame, \cite{GM-dendr}. Is it always tame$_\mathbf{1}$ ? 
	%	Is it true that for every continuous group action $T \times D \to D$ on a dendrite $D$ the $T$-space $X$ is Tame$_\mathbf{1}$ ? 
\end{question}

This is the case for the simplest nontrivial dendrite $D=[0,1]$. 
This can be proved by a minor modification of the arguments in Example \ref{ex:Helly}.
%2111
Note that the circle $S^1$, which is a local dendrite, does admit an action of $\mathrm{GL}(d,\R)$
which is tame but not tame$_{\bf 1}$ (Example \ref{e:EllisProj}). 
%0312  
See also Proposition \ref{ex:T}. 

\sk 

%3010
For other questions in the present work see:
%0607 Qustions
Questions 
\ref{pr:G-delta}, \ref{pr:RosComp=E}, \ref{q-rigid}, \ref{q:abel},
\ref{q:c-ord},  \ref{q:21},
\ref{pr:2-idemp}  and \ref{pr:rank}.

% \noindent\hrulefill
 
 \br 
 
\bibliographystyle{amsplain}

\begin{thebibliography}{10}

\bibitem{Ak-98}
E. Akin, {\em Enveloping linear maps\/}, in: Topological dynamics
and applications, Contemporary Mathematics {\bfseries 215}, a
volume in honor of R.~Ellis, 1998, pp. 121--131.

%0312 
\bibitem{AAB}
E. Akin, J. Auslander, and K. Berg,
{\em Almost equicontinuity and
	the enveloping semigroup\/}, Topological dynamics and
applications, Contemporary Mathematics  {\bfseries 215},
a volume in honor of R.~Ellis, 1998, pp. 75--81.


\bibitem{AAG}
E. Akin, J. Auslander and E. Glasner, 
{\em The topological dynamics of Ellis actions},
Mem. Amer. Math. Soc. {\bf 195},  (2008), no. 913.

\bibitem{A-D-K}
S.A. Argyros,  P. Dodos, and V. Kanellopoulos, 
{\em A classification of separable Rosenthal compacta and its applications},
Dissertationes Math. {\bf 449},  (2008), 52 pp.

\bibitem{Auj}
J.-B. Aujogue, 
{\em Ellis enveloping semigroup for almost canonical model sets of an Euclidean space},
Algebr. Geom. Topol. {\bf 15}, (2015), no. 4, 2195--2237.

%0312  
\bibitem{ABKL}
J.-B. Aujogue, M. Barge, J. Kellendonk and D. Lenz, 
\textit{Equicontinuous factors, proximality and Ellis semigroup for Delone sets}, 
Progress in Mathematics, v. 309, 137--194, Springer, 2015. 

\bibitem{Au}
J. Auslander, 
{\em Minimal flows and their extensions}, 
North-Holland Mathematics Studies, {\bf 153}, 
Notas de Matematica [Mathematical Notes], 122. 
North-Holland Publishing Co., Amsterdam, 1988.

\bibitem{BGKM}
F. Blanchard,  E. Glasner,  S. Kolyada and A. Maass, {\em On Li-Yorke pairs},
J. Reine Angew. Math. {\bf 547},  (2002), 51--68.

\bibitem{BJL}
M. Baake, T. J\"{a}ger, and D. Lenz. 
{\em Toeplitz flows and model sets},
Bull. Lond. Math. Soc., {\bf 48(4)},  (2016), 691--698.

\bibitem{Bour77}
J. Bourgain, 
{\em Compact sets of first Baire class},
Bull. Soc. Math. Belg. {\bf 29}, (1977), no. 2, 135--143.

\bibitem{Bour}
J. Bourgain, 
{\em Some remarks on compact sets of first Baire class}. 
Bull. Soc. Math. Belg. {\bf 30},  (1978), 3--10.

\bibitem{Bour-80}
J. Bourgain,
{\em On convergent sequences of continuous functions},
Bull. Soc. Math. Belg. {\bf 32},  (1980), 235--249.

\bibitem{Bo}
B.H. Bowditch, 
{\em Convergence groups and configuration spaces},
Geometric group theory down under (Canberra, 1996), 23--54, de Gruyter, Berlin, 1999.

\bibitem{Bo-98}
B.H. Bowditch, 
{\em A topological characterisation of hyperbolic groups},
J. Amer. Math. Soc. {\bf 11},  (1998), no. 3, 643--667. 

\bibitem{Debs}
G. Debs,
\emph{Descriptive aspects of Rosenthal compacta,} in: Recent Progress in General Topology III  
(Eds.: K.P. Hart, J. van Mill, P. Simon),  
Springer-Verlag, Atlantis Press, 2014. 

\bibitem{D}
T. Downarowicz, 
{\em Survey of odometers and Toeplitz flows}, 
Algebraic and topological dynamics,
Contemporary Mathematics 385 (eds. S. Kolyada, Y. Manin and T. Ward; American Mathematical Society,
Providence, RI, 2005) 7--37.

\bibitem{D-D}
T. Downarowicz and F. Durand, 
{\em Factors of Toeplitz flows and other almost 
%1811
%1 ? 1 
1-1
extensions over group rotations}, 
Math. Scand. {\bf 90},  (2002),  57--72.

\bibitem{D-S}
T. Downarowicz and J. Serafin, 
{\em Semicocycle extensions and the stroboscopic property}, 
Topology Appl. {\bf 153},  (2005), 97--106.

\bibitem{D-M}
E.B. Dynkin and M.B. Maljutov,
{\em Random walks on groups with finite number of generators},
Doll. Akad. Nauk. SSSR, {\bf 137}, (1961), 1042--1045.

\bibitem{E-book}
R. Ellis,
{\em Lectures on Topological Dynamics\/},
W. A. Benjamin, Inc.\ , New York, 1969.

\bibitem{E}
R. Ellis, {\it The enveloping semigroup of projective flows},
Ergod. Th. Dynam. Sys. {\bf 13} (1993), 635--660.

\bibitem{E-N}
R. Ellis and M. Nerurkar, {\em Weakly almost periodic flows\/},
Trans.\ Amer.\ Math.\ Soc.\ {\bfseries 313}, (1989), 103-119.

\bibitem{Eng}
R. Engelking,
{\em  General topology\/},
revised and completed edition,
Heldermann Verlag, Berlin, 1989.

\bibitem{F}
E.E. Floyd,
{\em  A nonhomogeneous minimal set},
Bull. Amer. Math. Soc. {\bf 55}, (1949), 957--960.

\bibitem{FS}
J.J. Font and M. Sanchis,  
\textit{Sequentially compact subsets and monotone 
functions: an application to Fuzzy Theory,} 
Topology and its Applications, 
V. 192, 2015, Pages 113--122.

\bibitem{FGJO}
G. Fuhrmann, E. Glasner, T. J\"{a}ger and  C. Oertel,
{\em Irregular model sets and tame dynamics},
arXiv:1811.06283. 

%1211 writing: Kellendonk 
\bibitem{F-K-Y}
G. Fuhrmann, J. Kellendonk and R. Yassawi,
{\em Tame or wild Toeplitz shifts},
arXiv:2010.11128v1. 

\bibitem{F-K}
G. Fuhrmann and D. Kwietniak,
{\em  On tameness of almost automorphic dynamical systems for general groups},
Bulletin of the London Mathematical Society, {\bf 52} (1), (2020),  24--42.

\bibitem{Fu}
H. Furstenberg, 
{\em Disjointness in ergodic theory, minimal sets, and a problem in
Diophantine approximation\/},
Math.\ System Theory {\bfseries 1}, (1967), 1--49.

\bibitem{G}
S. Glasner,
{\em Proximal flows\/},
Lecture Notes in Math.\ {\bfseries 517}, Springer-Verlag, 1976.

\bibitem{G-env} 
E. Glasner, {\it Enveloping semigroups in topological dynamics}, 
Topology and Appl., \textbf{154} (2007), 2344--2363. 

\bibitem{G-06}
E. Glasner,
{\em On tame dynamical systems},
Colloq. Math. {\bf 105},  (2006), no. 2, 283--295.

\bibitem{G-18}
E. Glasner,
{\em The structure of tame minimal dynamical systems for general groups},
Invent. math. {\bf 211}, (2018), 213--244.

\bibitem{G-G}
E. Glasner and Y. Glasner, 
{\em A metric minimal PI cascade with $2^{\cb}$  minimal ideals},
Ergodic Theory Dynam. Systems, {\bf 40},  (2020), no. 5, 1268--1281. 

\bibitem{G-K}
E. Glasner and J.L. King, 
{\em A zero-one law for dynamical properties},
Topological dynamics and applications (Minneapolis, MN, 1995), 231--242,
Contemp. Math., {\bf 215}, Amer. Math. Soc., Providence, RI, 1998.

%0312 
\bibitem{G-Ma}
E. Glasner and D. Maon,
{\em Rigidity in topological dynamics\/},
Ergod.\ Th.\ Dynam.\ Sys.\  {\bfseries 9},
(1989), 309--320.

\bibitem{G-M}
E. Glasner, and M. Megrelishvili,
{\em Hereditarily non-sensitive dynamical systems and linear representations},
Colloq. Math. {\bf 104},  (2006), no. 2, 223--283.



%\bibitem{GM-suc}
%E. Glasner and M. Megrelishvili, {\it New algebras of functions on
%	topological groups arising from $G$-spaces}, Fundamenta Math., \textbf{201}
%(2008), 1--51.

\bibitem{GM-rose}
E. Glasner and M. Megrelishvili, {\it Representations of dynamical
	systems on Banach spaces not containing $l_1$},
Trans. Amer. Math. Soc.,  \textbf{364} (2012),  6395--6424.
%ArXiv e-print: 0803.2320.

%\bibitem{GM-fp}
%E. Glasner and M. Megrelishvili, {\em On fixed point theorems and
%	nonsensitivity},
%Israel J. of Math., \textbf{190} (2012), 289--305.
%%ArXiv e-print: 1007.5303.

%\bibitem{GM-AffComp}
%E. Glasner and M. Megrelishvili,
%\emph{Banach representations and affine compactifications of dynamical systems},
%in: Fields institute proceedings dedicated to the 2010 thematic program on asymptotic geometric analysis,
%M. Ludwig, V.D. Milman, V. Pestov, N. Tomczak-Jaegermann (Editors), Springer, New-York, 2013. 
%ArXiv version: 1204.0432.



%\bibitem{GM-fp}
%E. Glasner and M. Megrelishvili, {\em On fixed point theorems and
%	nonsensitivity},
%Israel J. of Math., \textbf{190} (2012), 289--305.
%%ArXiv e-print: 1007.5303.

\bibitem{GM-AffComp}
E. Glasner and M. Megrelishvili,
\emph{Banach representations and affine compactifications of dynamical systems},
in: Fields institute proceedings dedicated to the 2010 thematic program on asymptotic geometric analysis,
M. Ludwig, V.D. Milman, V. Pestov, N. Tomczak-Jaegermann (Editors), Springer, New-York, 2013. 
ArXiv version: 1204.0432.

\bibitem{GM-survey}
E. Glasner and M. Megrelishvili,
\emph{Representations of dynamical systems on Banach spaces,}
in: Recent Progress in General Topology III, 
(Eds.: K.P. Hart, J. van Mill, P. Simon),  
Springer-Verlag, Atlantis Press, 2014, 399--470. 

\bibitem{GM-c}
E. Glasner and M. Megrelishvili,
\emph{Circularly ordered dynamical systems,} %ArXiv:1608.05091, 2016, 
Monats. Math. \textbf{185} (2018), 415--441.

\bibitem{GM-tLN}
E. Glasner and M. Megrelishvili,
\emph{More on tame dynamical systems}, in: Lecture Notes S. vol. 2013, Ergodic Theory and Dynamical Systems in their Interactions with Arithmetics and Combinatorics, Eds.: S. Ferenczi, J. Kulaga-Przymus, M. Lemanczyk,  Springer, 2018, pp. 351--392.  



\bibitem{GM-dendr}
E. Glasner and M. Megrelishvili, \textit{Group actions on treelike compact spaces}, Science
China Math. \textbf{62} (2019). 



%\bibitem{G-M-06}
%E. Glasner, and M. Megrelishvili,
%{\em More on tame dynamical systems},
%Ergodic theory and dynamical systems in their interactions with arithmetics and combinatorics, 351--392,
%Lecture Notes in Math., {\bf 2213}, Springer, Cham, 2018.

\bibitem{GM-int} 
E. Glasner and M. Megrelishvili,
\textit{Circular orders, ultrahomogeneity and topological groups}, ArXiv:1803.06583. To appear in: AMS Contemporary Mathematics book series volume "Topology, Geometry, and Dynamics: Rokhlin - 100".

\bibitem{G-M-U}
E. Glasner, M. Megrelishvili and V.V. Uspenskij, 
{\em On metrizable enveloping semigroups},
 Israel J. Math. {\bf 164},  (2008), 317--332.

%1811
%\bibitem{Gli} 
%I. Glicksberg, \textit{Stone-Cech compactifications of products}, Trans. Amer. Math., \textbf{90}
%(1959), 369--389.

%0312 
 \bibitem{GW}
E. Glasner and B. Weiss,
{\em Sensitive dependence on initial conditions\/},
Nonlinearity\  {\bfseries 6},
(1993), 1067--1075.

\bibitem{H-J}
K.N. Haddad and A.S. Johnson, 
{\em Auslander systems},
Proc. Amer. Math. Soc. {\bf 125}, (1997), no. 7, 2161--2170.

\bibitem{H}
W. Huang, {\em Tame systems and scrambled pairs under an abelian
group action\/}, Ergod. Th. Dynam. Sys. {\bf 26} (2006),
1549--1567.

 \bibitem{diss:Jolivet} 
T. Jolivet, 
\emph{Combinatorics of Pisot Substitutions}, TUCS Dissertations No 164, 2013. 

\bibitem{Koh}
A. K\"{o}hler, {\em Enveloping semigrops for flows}, Proc.
of the Royal Irish Academy, {\bf 95A} (1995), 179--191.

%1811
\bibitem{Ke-L}
 A.S. Kechris and A. Louveau, 
 {\em A classification of Baire class 1 functions},
  Trans. Amer. Math.
Soc. {\bf 318}, (1) (1990), 209--236.

\bibitem{K-L}
D. Kerr and H. Li,
{\em Independence in topological and $C^*$-dynamics},
 Math. Ann. {\bf 338}, (2007), no. 4, 869--926.
 
 \bibitem{LB}  
D.J. Lutzer, H.R. Bennett, 
 \textit{Separability, the countable chain condition and the 
 Lindelof property in linearly orderable spaces}, Proc. AMS, 23 (1969), pp. 664--667. 
 
 %060521 
 \bibitem{Mar}
 W. Marciszewski, \textit{Modifications of the double arrow space and
 related Banach spaces $C(K)$,} 
Studia Math. 184:3 (2008), 249--262.  
 
 %060521 
 \bibitem{Ost}
 A.J. Ostaszewski, \textit{A characterization of compact, separable, ordered spaces}, J. London Math. Soc. 7 (1974), 758--760.
 
\bibitem{Pa}
M.E. Paul, 
{\em Construction of almost automorphic symbolic minimal flows},
General Topology and Appl. {\bf 6},  (1976), no. 1, 45--56.

 \bibitem{P}
R. Pol,
{\em On weak and pointwise topology in functions spaces},
Preprint, University of Warsaw (1984). 

\bibitem{Ro} H.P. Rosenthal,
\emph{A characterization of Banach spaces containing $l_1$}, Proc.
Nat. Acad. Sci. U.S.A., \textbf{71} (1974), 2411--2413.

\bibitem{S}
L. Shapiro,
{\em Proximality in minimal transformation groups},
Proc. Amer. Math. Soc. {\bf 26},  (1970), 521--525.

\bibitem{SS}
L.A. Steen, J.A. Seebach, \textit{Counterexamples in Topology}, Dover Publications, New York, 1978. 

%\bibitem{T-b}
%S. Todor\u{c}evi\'{c},
%{\em Topics in topology},
%Springer-Verlag, Lecture Notes in Mathematics,
%{\bfseries 1652}, 1997.

\bibitem{T}
S. Todor\u{c}evi\'{c},
{\em Compact subsets of the first Baire class},
J. of the AMS, {\bf 12}, (1999), 1179--1212.

\bibitem{V65}
W.A. Veech,
{\em Almost automorphic functions on groups\/},
Amer. J. Math.\ {\bfseries 87}, (1965), 719-751.

\bibitem{Y-Z}
X. Ye and R. Zhang,
{\em On sensitive sets in topological dynamics},
Nonlinearity {\bf 21},  (2008) 1601--1620.

\end{thebibliography}

\end{document}